\documentclass[10pt]{amsart}
\usepackage{amsmath}
\usepackage{amsfonts}
\usepackage{amsthm}
\usepackage{amssymb}
\usepackage{url}
\usepackage{fouriernc}
\usepackage{graphicx}
\usepackage{hyperref}
\usepackage[usenames,dvipsnames,svgnames,table]{xcolor}
\definecolor{badgerred}{rgb}{0.715,0.004,0.004}
\definecolor{burntorange}{rgb}{0.801,0.332,0.0}
\usepackage[font=small,labelfont={bf,sf}, textfont={sf},margin=0em]{caption}
\hypersetup{colorlinks,
  filecolor=black,
  linkcolor=MidnightBlue,
  citecolor=NavyBlue,
  urlcolor=RoyalBlue,
  bookmarksopen=true}
\hypersetup{bookmarksopenlevel=1}

\theoremstyle{definition}
\newtheorem{definition}{Definition}[section]
\newtheorem{thm}{Theorem}[section]
\newtheorem*{theorem*}{Theorem}
\newtheorem*{acknowledgement*}{Acknowledgements}
\newtheorem*{definition*}{Definition}

\newtheorem{lem}[thm]{Lemma}

\numberwithin{equation}{section}

\newcommand{\isdef}{{\stackrel{\textrm{def}}=}}
\newcommand{\mat}{\begin{pmatrix}} \newcommand{\rix}{\end{pmatrix}}
\newcommand{\tmat}{\left(\begin{smallmatrix}} \newcommand{\trix}{\end{smallmatrix}\right)}
\newcommand{\deter}{\left|\begin{matrix}} \newcommand{\minant}{\end{matrix}\right|}

\newcommand{\N}{\mathbb{N}}
\newcommand{\R}{\mathbb{R}}
\newcommand{\vt}{\vartheta}

\newcommand{\ve}[1]{{{\boldsymbol e}_{#1}}}
\newcommand{\vE}[1]{{{\boldsymbol E}_{#1}}}

\newcommand{\cA}{{\mathcal A}}
\newcommand{\cC}{{\mathcal C}}
\newcommand{\cE}{{\mathcal E}}
\newcommand{\cL}{{\mathcal L}}
\newcommand{\cO}{{\mathcal O}}
\newcommand{\cR}{{\mathcal R}}

\definecolor{orange}{RGB}{250, 140, 0}

\definecolor{turq}{RGB}{0, 160, 160}

\renewcommand{\emptyset}{\varnothing}	
\renewcommand{\tilde}[1]{\widetilde{#1}}

\title[Nonconvex ancient solutions]{Nonconvex ancient solutions\\ to Curve Shortening Flow}

\author[Zhang]{Yongzhe Zhang}
\address{Yongzhe Zhang}
\email{yongzhe@caltech.edu}
\thanks{The first author was partially supported by the NSF grants DMS-2018220 and DMS-2018221.}

\author[Olson]{Connor Olson}
\address{Connor Olson}
\email{cjo5325@psu.edu}
\thanks{The second author is supported by the National Science Foundation Graduate Research Fellowship Program under Grant No.~DGE1255832.
Any opinions, findings, and conclusions or recommendations expressed in this material are those of the authors and do not necessarily reflect the views of the National Science Foundation.}

\author[Khan]{Ilyas Khan}
\address{Ilyas Khan \\ University of Oxford \\ Mathematical Institute \\Andrew Wiles Building\\ Woodstock Rd \\ Oxford, UK \\  OX2 6GG} 
\email{Ilyas.Khan@maths.ox.ac.uk}
\thanks{The third author was partially supported by the Simons Collaboration on Special Holonomy in Geometry, Analysis, and Physics (\#724071 Jason Lotay).}

\author[Angenent]{Sigurd Angenent}
\address{Sigurd Angenent, University of Wisconsin--Madison, Van Vleck Hall, Madison, WI53706, USA}
\email{angenent@wisc.edu}

\begin{document}

\begin{abstract}
We construct an ancient solution to planar curve shortening.  The solution is at all times compact and embedded.  For \(t\ll0\) it is approximated by the rotating Yin-Yang soliton, truncated at a finite angle \(\alpha(t) = -t\), and closed off by a small copy of the Grim Reaper translating soliton.
\end{abstract}
\maketitle

\section{Introduction}
In \cite{DaskalopoulosHamiltonSesum:CompactAncientCS} Daskalopoulos, Hamilton, and Sesum showed that any compact, convex, and embedded ancient solution to Curve Shortening in the plane is either a shrinking circle or the ancient paperclip solution.  Qian You et.al.\cite{QianYouThesis,QianYouAngenent2021} showed that there exist many other ancient solutions that are either embedded and not compact, or otherwise compact, convex, but not embedded.  Here we construct an ancient solution to Plane Curve Shortening that is embedded, compact at all times, but not convex. We note that an ancient solution satisfying these properties was also constructed independently in \cite{CharyyevAncientSpirals}.  We present a brief comparison of the two different approaches at the end of this introduction.

Our construction begins with the Yin--Yang soliton, i.e.~the rotating soliton that is invariant with respect to reflection in the origin---see Figure~\ref{fig-truncated-YY} (left).  This spiral shaped curve divides the plane into two congruent parts, and under Curve Shortening evolves by rotating with unit speed in the counterclockwise direction.
\begin{figure}[b]
\end{figure}
Each of the two branches of the Yin--Yang spiral is a graph in polar coordinates, given by
\[
r=\cR(\theta+\tfrac\pi2), \quad r=\cR(\theta-\tfrac\pi2)
\]
respectively.  The Yin-Yang spiral is asymptotic to a Fermat spiral \cite{enwiki:1025901425} which in polar coordinates is given by \(r=a\sqrt\theta\).  The Yin-Yang curve satisfies
\begin{equation}
\label{eq-YY-simple-asymptotics}
\cR(\theta) = \sqrt{2\theta} + \cO(\theta^{-3/2}), \qquad
\cR'(\theta) = \frac1{\sqrt{2\theta}} + \cO(\theta^{-5/2})\quad (\theta\to\infty).
\end{equation}
We review the properties and derive expansions for \(\cR\) in appendix~\ref{sec-YY-properties}.

The Yin--Yang soliton is itself an example of an embedded ancient (and in fact eternal) solution.  It is however not compact.  Our goal in this paper is to construct a compact embedded ancient solution which, for \(t\to-\infty\) converges to the rotating Yin--Yang solution.
\begin{figure}[h]
\includegraphics[width=0.45\textwidth]{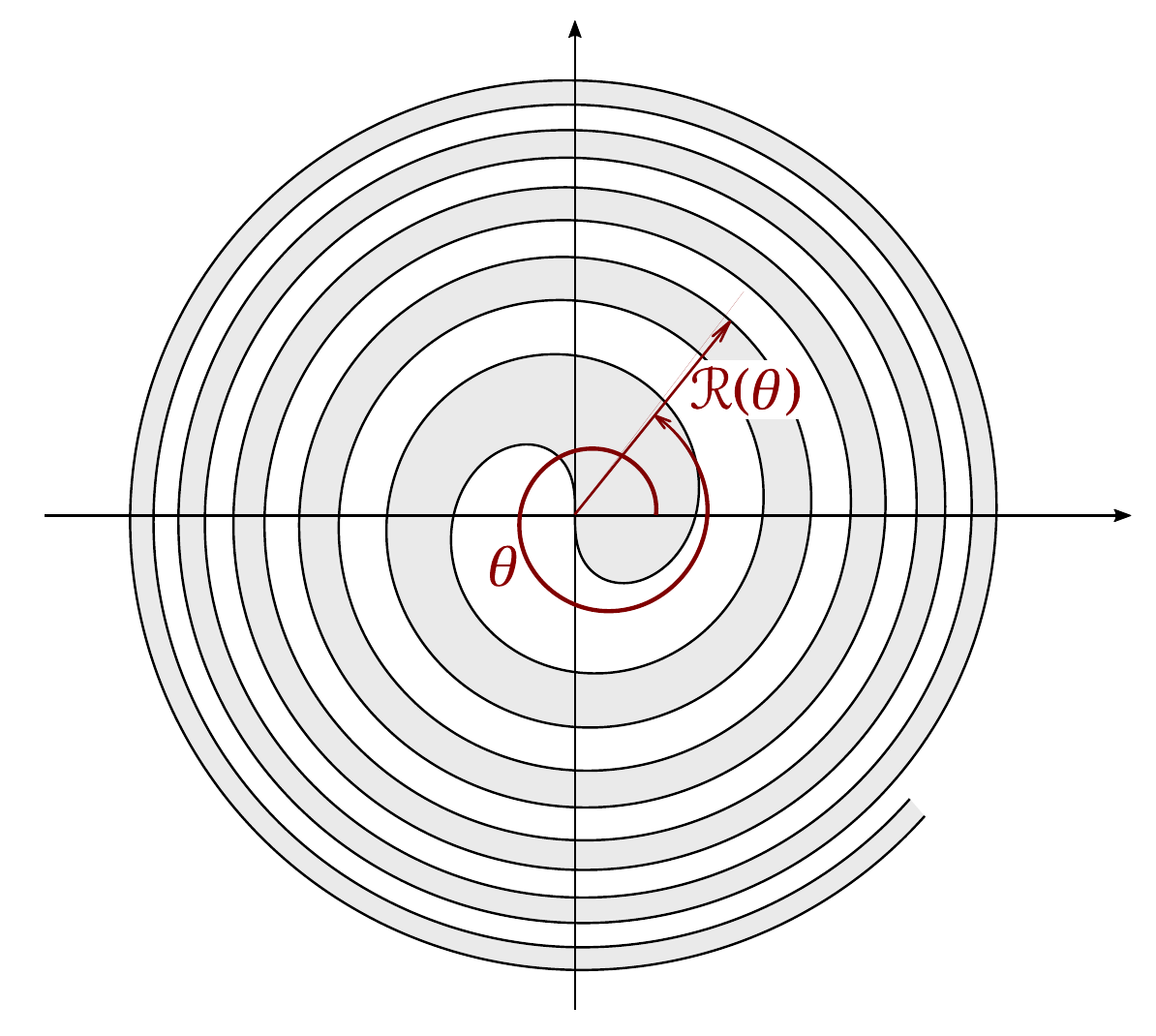}
\includegraphics[width=0.45\textwidth]{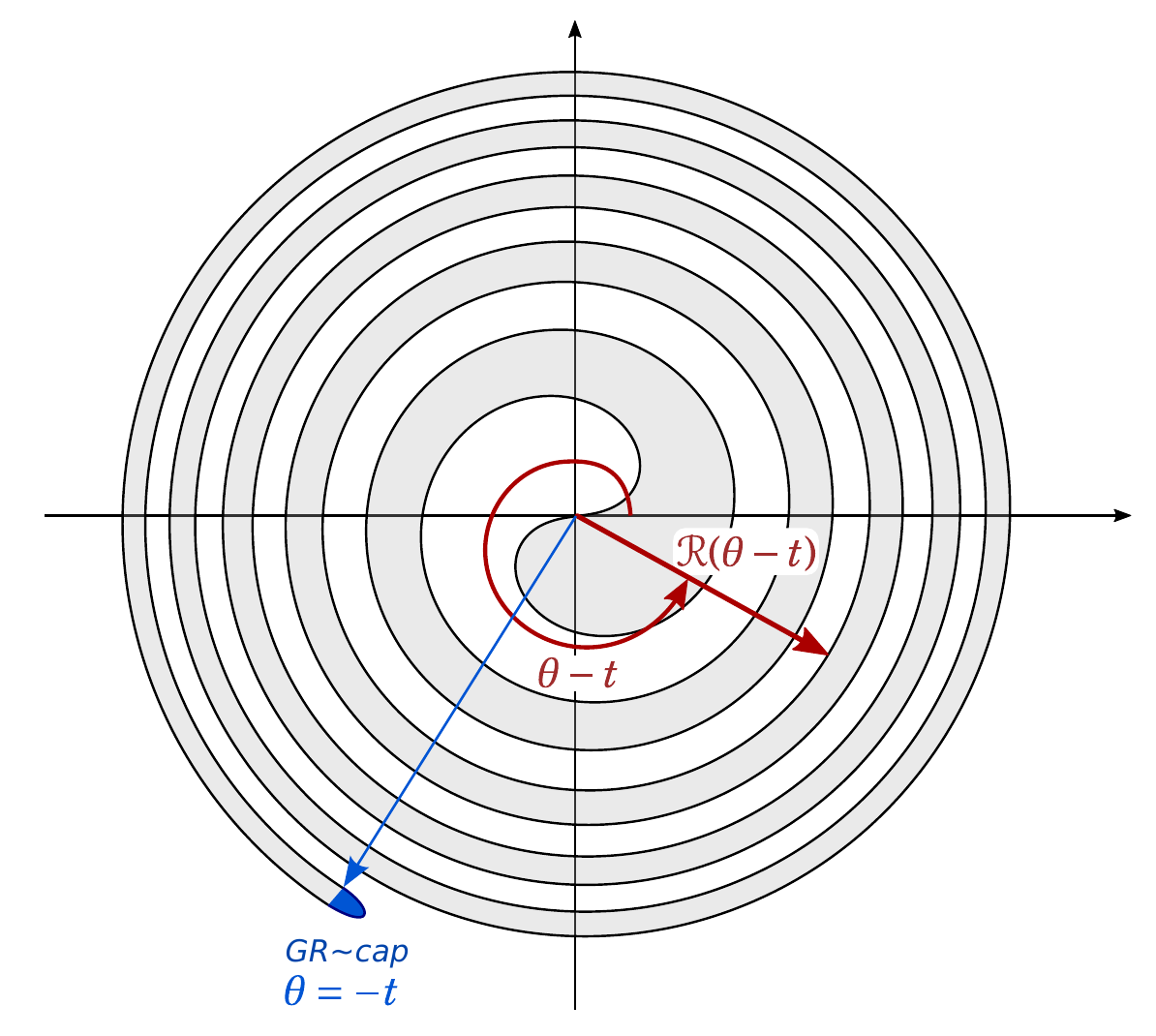}
\caption{ \textbf{Left:} The Yin Yang soliton in polar coordinates.  \textbf{Right:} The Truncated Yin Yang soliton at time \(t\ll 0\).  The bulk of the curve consists of an initial segment of the Yin--Yang curve rotated counterclockwise by an angle \(t\), which has two ends.  These ends are connected by a small cap in the shape of a translating soliton (the so-called ``Grim Reaper'') located at polar angle \(\theta=-t\).}
\label{fig-truncated-YY}
\end{figure}
In section~\ref{sec-cap-construction} of this paper we construct an approximate solution to CSF, by truncating the Yin--Yang solution at polar angle \(\alpha(t)\), and connecting the two ends with a small cap whose shape is approximately that of a rotated and rescaled copy of the Grim Reaper soliton.  See Figure~\ref{fig-truncated-YY} (right).

To determine how to choose the angle \(\alpha(t)\) we estimate the speed \(V\) of the cap at time \(t\).  Since the cap is approximately a Grim Reaper, its speed is related to its width \(w\) by \(V = \pi/w\).  The width \(w\) is approximated by
\[
w\approx \cR(\alpha(t)-t+\tfrac\pi2) - \cR(\alpha(t)-t-\tfrac\pi2) \approx \cR'(\alpha(t)-t)\,\pi,
\]
so that the cap moves with speed
\[
V\approx \frac\pi w \approx \frac 1{\cR'(\alpha(t)-t).}
\]
On the other hand, far away from the center, the arms of the Yin-Yang spiral are close to circular.  The cap, which moves with angular velocity \(-\alpha'(t)\) along a near circle with radius \(\cR(\alpha(t)-t)\), therefore has velocity \(V\approx -\alpha'(t)\cR(\alpha(t)-t)\).  It follows that \(\alpha'(t) \approx -1/(\cR\cR')\).  Since the asymptotics~\eqref{eq-YY-simple-asymptotics} imply \(\cR\cR'\approx 1\), we end up with
\begin{equation}
\label{eq-alpha-heuristic-choice}
\alpha'(t) = -1+o(1), \qquad \alpha(t) = -t+o(t).
\end{equation}

For any \(t<0\) we let \(\Omega(t)\subset\R^2\) be the region given in Polar Coordinates by
\begin{equation}
\label{eq-Omega-defined}
\Omega(t) = \Bigl\{(r\cos\theta, r\sin \theta) \;\big|\; t\leqslant\theta\leqslant -t,
\;\cR(\theta-t-\tfrac\pi2) \leqslant r\leqslant \cR(\theta-t+\tfrac\pi2)\Bigr\}
\end{equation}
We will call the boundary curve \(\partial\Omega(t)\) the \emph{truncated Yin-Yang curve}.  It consists of a segment of the rotating Yin-Yang curve and a straight line segment that connects the two ends at \(\theta=-t\).

\medskip\noindent \textbf{Main Theorem.  } There exists a compact ancient solution \(\{C(t) \mid t<0\}\) to Curve Shortening which for \(t\to-\infty\) is uniformly close to the truncated Yin-Yang curve in the sense that if \(\Lambda(t)\) is the bounded region enclosed by \(C(t)\) then\footnote{For two sets \(A\) and \(B\) we denote their symmetric difference by \(A\triangle B = (A\setminus B)\cup(B\setminus A)\).}
\[
\mathrm{Area}(\Lambda(t) \triangle \Omega(t)) \lesssim |t|^{-1+\delta} \qquad (t\to-\infty)
\]
for every \(\delta>0\).

\bigskip

The construction follows a pattern similar to the construction of ancient solutions in \cite{QianYouAngenent2021,QianYouThesis}, namely, construct a sequence of ``really old solutions'' \(\{\cC_n(t) \mid -n\leqslant t\leqslant t_0\}\) of Curve Shortening and extract a convergent subsequence whose limit is the desired ancient solution.  In section~\ref{sec-cap-construction} we first construct an ancient approximate solution \(\{ \cC_*(t) \mid -\infty < t< 0\}\) of Curve Shortening, i.e.~a family of curves for which
\begin{equation}
\label{eq-error-def}
\cE \isdef
\int_{-\infty}^{t_0}\int_{\cC_*(t)}  \left| V -\kappa\right|\, ds\, dt <\infty .
\end{equation}
Then we consider a sequence of solutions \(\{\cC_n(t) \mid -n\leqslant t <t_0\}\) of Curve Shortening whose initial curves \(\cC_n(-n)\) are chosen increasingly close to the approximate solution \(\cC_*(-n)\) at time \(-n\), in the sense that the area between \(\cC_*(-n)\) and \(\cC_n(-n)\) tends to zero as \(n\to\infty\).  Arguing as in \cite{QianYouAngenent2021,QianYouThesis} we observe that the area between the solutions \(\cC_n(t)\) and the approximate solution \(\cC_*(t)\) is bounded in terms of the ``error'' \(\cE\) and the area between the initial curves \(\cC_*(-n)\) and \(\cC_n(-n)\).  Since the curves we deal with in this paper are not graphical the area estimate is a bit more complicated than in \cite{QianYouAngenent2021,QianYouThesis}.  In section~\ref{sec-area-growth-bound} we present a more general estimate that generalizes the Altschuler-Grayson \cite{MR1158337} area bounds for Space Curve Shortening.  Finally, in section~\ref{sec-convergence} we show how this bound allows us to extract a convergent subsequence of the very old solutions \(\cC_n(t)\), and provides enough control to conclude that the resulting limit satisfies the description in the Main Theorem.

In appendix~\ref{sec-YY-properties} we recall the derivation of the Yin-Yang soliton, and obtain its asymptotic expansion at infinity.

\textbf{Comparison with the construction in \cite{CharyyevAncientSpirals}. }
The idea to construct ancient solutions as a limit of a sequence of very old solutions is natural.  In \cite{CharyyevAncientSpirals} the same strategy is also followed, but the core of the proof where one controls the sequence of old solutions is quite different. While the approach in our paper follows the method of area comparison established in \cite{QianYouAngenent2021}, the arguments in \cite{CharyyevAncientSpirals} proceed by carefully controlling a sequence of old solutions via an exponential barrier for the bulk of the solution.  Thereafter stability of the tip is established in \cite{CharyyevAncientSpirals} using a blow-up argument, combined with the uniqueness of the Grim Reaper as possible limit.

\section{Construction of the cap}\label{sec-cap-construction}
\subsection{Parametrized curves and the Curve Shortening deficit}
An evolving family of curves is a map \(X:(t_0, t_1) \times \R\to\R^2\) for which \(X_p(t, p)\neq 0\) for all \((t,p)\).  For such a family we define
\[
ds = \|X_p(t, p)\|\, dp,\qquad \frac{d}{ds} = \frac{1}{\|X_p(t, p)\|} \frac{d}{dp}\;.
\]
The normal velocity and curvature of the family \(X\) are
\[
V = \left\langle X_t, JX_s\right\rangle, \qquad \kappa = \left\langle X_{ss}, JX_s\right\rangle =\left\langle \frac{X_{pp}}{\|X_p\|^2}, \frac{JX_p}{\|X_p\|}\right\rangle
\]
where \( J= \tmat 0&-1 \\ 1 &0\trix\) represents counterclockwise rotation by \(\frac\pi2\).  By definition, the parameterized family of curves \(X\) satisfies Curve Shortening if \(V=\kappa\).  If it does not, then we measure the ``discrepancy with curve shortening'' in terms of the form
\begin{equation}
\label{eq-discrepancy-def}
|V-\kappa|\, ds
= \left|\Bigl\langle X_t-\frac{X_{pp}}{\|X_p\|^2}, JX_p\Bigr\rangle\right| \, dp
\end{equation}
The error \(\cE\) is obtained by integrating this form over the curve and in time.

We write \(\ve1 = (1,0)\), \(\ve2=(0,1)\) for the standard basis for \(\R^2\).  In this basis counterclockwise rotation by \(\theta\) is given by
\[
e^{\theta J} = \mat \cos \theta & -\sin \theta \\ \sin \theta & \cos \theta \rix.
\]
We will use the rotated frame \(\{\vE1(\theta), \vE2(\theta)\} = \{e^{\theta J}\ve1 , e^{\theta J}\ve2\}\), which satisfy
\[
\vE1'(\theta) = \vE2(\theta), \qquad \vE2'(\theta) = -\vE1(\theta).
\]

\subsection{\(Z\) or \((u, v)\)-coordinates}
We expect the tip of the ancient solution to be located near the point \(R(t)\vE1(-t)\), and to have width \(\sim R^{-1}\), where
\[
R=R(t)\isdef \cR(-2t).
\]
Thus we introduce new, time dependent, coordinates \(Z = u\ve1+v\ve2\) related to the cartesian coordinates \(X=(x_1, x_2)\) via
\begin{equation}
\label{eq-local-Z-coords}
X =  e^{-tJ}\{R\ve1 + R^{-1}Z\}, \text{ i.e.  }
Z = R \bigl\{e^{tJ}X - R\ve1\bigr\}.
\end{equation}

\subsection{The inner and outer Yin-Yang arms}

The region on one side of the Yin-Yang curve is foliated by rotated copies of the curve.  At time \(t\in\R\) the leaves of this foliation are parametrized by
\begin{equation}
\label{eq-YY-leaves-parametrized}
Y(\theta, t, y) =   \cR(\theta-t+y) \vE1(\theta) 
\end{equation}
where \(y\in[-\pi/2, \pi/2]\) determines the leaf, and \(\theta\in(t-y, \infty)\) is the polar angle on the leaf.  The inner and outer arms of the region that contains our ancient solution correspond to \(y=\pm\pi/2\).  See Figure~\ref{fig-foliation}.

\begin{figure}[t]
\centering
\includegraphics[width=0.45\textwidth]{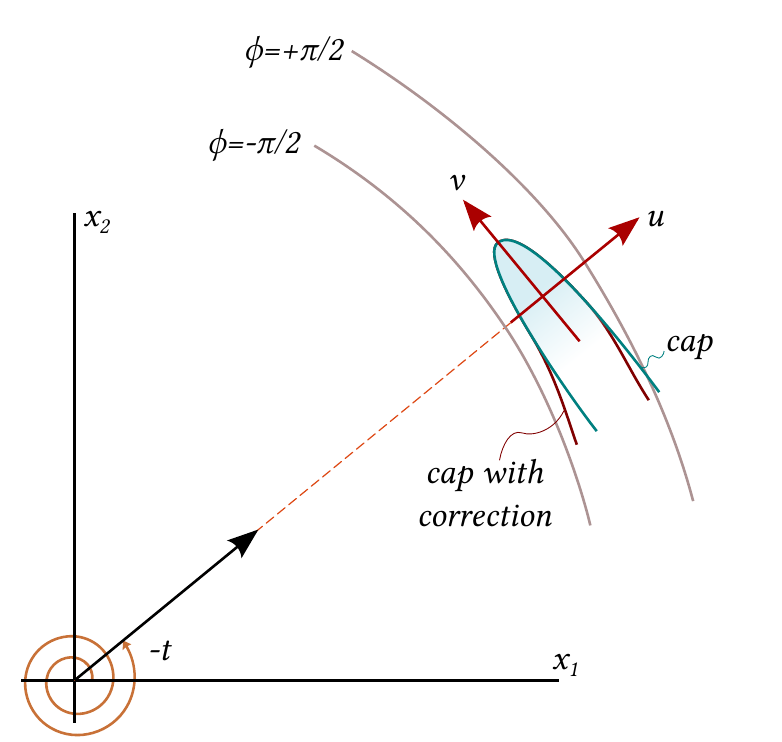}\hfill
\includegraphics[width=0.5\textwidth]{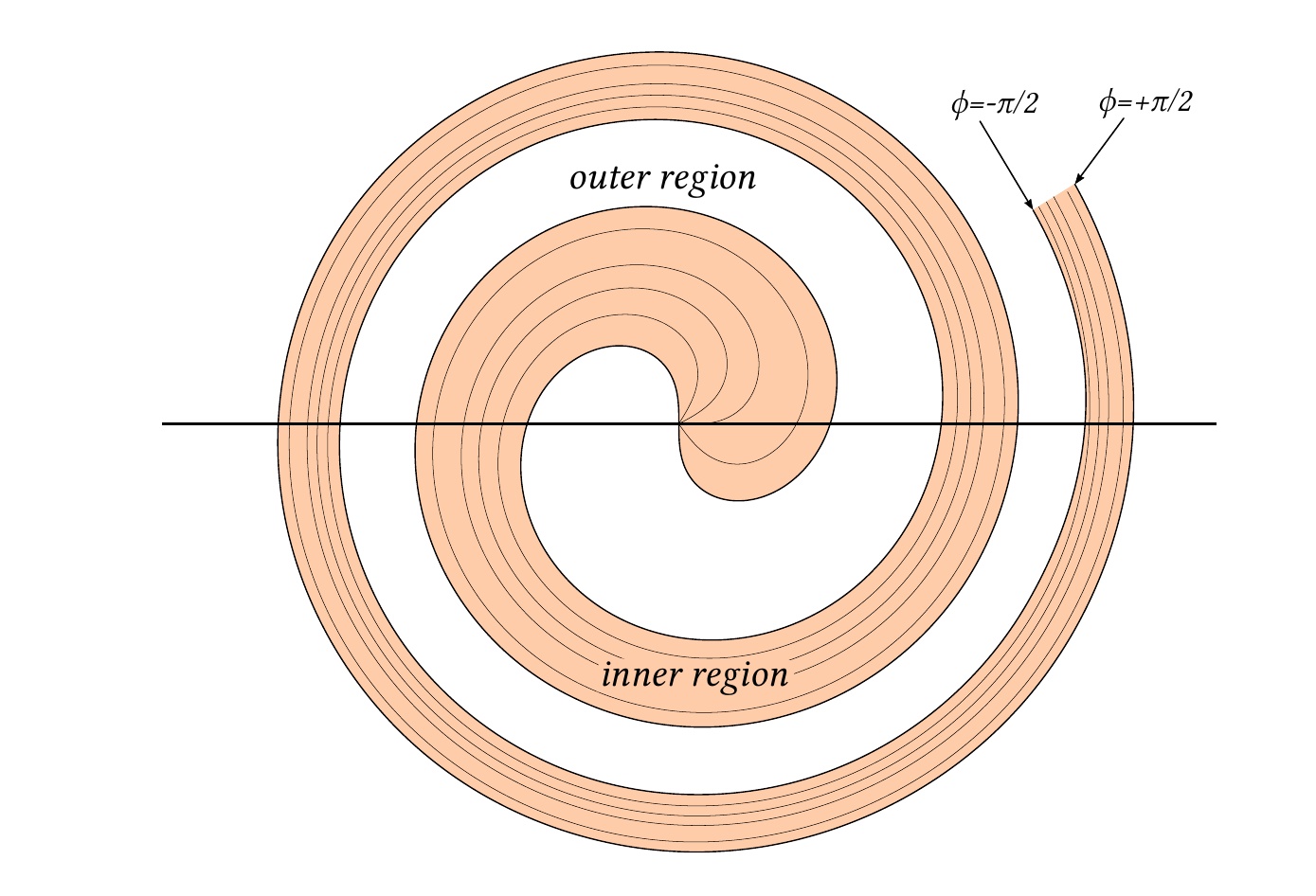}
\caption{\textbf{Left:} The \((u,v)\) coordinates.\qquad \textbf{Right:} The Yin Yang foliation at time \(t=0\).}
\label{fig-foliation}
\end{figure}

\subsection{Lemma}\label{subsec-yinyang-local-coords-lemma}
For any \(M>0\) there is a \(t_M<0\) such that if \(t\leqslant t_M\), then the segments of the Yin-Yang leaves \(Y(\theta, t, y)\) with \(|\theta+t|\leqslant M\tau^{-1} \ln\tau\) are graphs in \((u,v)\) coordinates of the form \(u=U_y(t, v)\), at least if \(t \leqslant t_M\).  Moreover, the functions \(U_y(t, v)\) satisfy
\begin{equation}
\label{eq-leaves-in-uv-expansion}
U_y(t, v) = y - \frac{y^2+v^2-2v}{2\tau} + \cO(\tau^{-2+\delta})
\end{equation}
for all $\delta >0$, where \(\tau=-4t\).
\begin{proof}
We begin with the defining equations
\begin{equation}\label{eq-YY-coords0}
\cR(\theta-t+y) \vE1(\theta) = e^{-tJ} \bigl\{ R \ve1 +R^{-1}Z\bigr\}
=e^{-tJ}\bigl\{(R+u/R)\ve1 + v/R\, \ve2\bigr\}.
\end{equation}
Multiply with \(e^{tJ}\) on both sides:
\[
\cR(\theta-t+y) \vE1(\theta+t) = R \ve1 +R^{-1}Z = (R+u/R)\ve1 + v/R\, \ve2.
\]
Here the left hand side is the polar form of the right hand side.  Under our assumptions \(R^2\sim \tau\) and \(|u|\ll \tau\), so \(1+u/R^2>0\), and hence we have
\begin{equation}\label{eq-YY-coords1}
\cR(\theta-t+y)^2 = R^2 + 2u + \frac{u^2+v^2}{R^{2}},\qquad
\theta+t = \arctan \frac{v/R^2}{1+u/R^2}.
\end{equation}
Since \(\alpha\mapsto \cR(\alpha)^2\) is a monotone function, the first equation in \eqref{eq-YY-coords1} can be solved for \(y\).  Using the asymptotic expansion in appendix~\ref{sec:cR-inverted} we get
\begin{equation}
2(\theta-t+y) = \cR^2 + 2\cR^{-2} +\cO(\cR^{-4}),
\end{equation}
where \(\cR=\cR(\theta-t+y)\).  Replace \(\cR^2\) by the expression in \eqref{eq-YY-coords1}, to get
\[
2(\theta-t+y) = R^2 + 2u + \frac{u^2+v^2+2}{R^2} +\cO\bigl(v^2R^{-4}\bigr).
\]
To eliminate \(\theta\) we expand the second equation in \eqref{eq-YY-coords1},
\[
\theta+t = \frac{v}{R^2} + \cO\bigl(v^2R^{-4}\bigr).
\]
Hence
\begin{align}\label{eq-YY-coords2}
  2y &= 2(\theta-t+y)-2(\theta+t)+4t & (-4t=\tau)\\
        &= R^2-\tau + 2u
          + \frac{u^2+v^2-2v+2}{R^2}+\cO\bigl(v^2R^{-4}\bigr).
          \notag
\end{align}
Expand \(R=R(t)=\cR(-2t)\) in powers of \(\tau=-4t\) using~\eqref{eq-R-etc-expanded}:
\[
R^2 = \cR(-2t)^2 = \tau \left(1-2\tau^{-2} + \cO(\tau^{-3})\right)
\]
and substitute this in \eqref{eq-YY-coords2} to get
\[
2y = 2u +\frac{u^2+v^2-2v}{\tau} + \cO(v^2R^{-4}).
\]
By assumption we have \(|v|\lesssim \ln\tau\) and \(R\sim \sqrt\tau\), so \(v^2R^{-4} =\cO(\tau^{-2+\delta})\) for any \(\delta>0\).

If \(\tau\) is sufficiently large then the above equation has a unique solution \(u = U_y(t, v)\) with
\[
u = y - \frac{u^2+v^2-2v}{2\tau} + \cO(\tau^{-2+\delta}) = y - \frac{y^2+v^2-2v}{2\tau} + \cO(\tau^{-2+\delta}).\qedhere
\]
\end{proof}

\subsection{General \emph{ansatz} for the cap}
We now construct the cap by assuming that it is given by
\begin{equation}
\label{eq-cap}
X(t, p) =  e^{-tJ}\bigl\{R\ve1 + \varepsilon Z(t, p)\bigr\},
\text{ with } R(t) \isdef \cR(-2t),\; \varepsilon(t) = \frac{1}{R(t)},
\end{equation}
and by computing its deviation from Curve Shortening~\eqref{eq-discrepancy-def}
\[
(V-\kappa)ds = \langle X_t-X_{ss}, JX_s\rangle ds = \Bigl\langle X_t-\frac{X_{pp}}{\|X_p\|^2}, JX_p \Bigr\rangle \;dp \isdef W\, dp
\]
In the following computations it will be convenient to abbreviate
\[
R_\theta(t) = \cR'(-2t),
\]
so that \(R'(t) = -2 R_\theta(t)\).

The space derivatives of \(X\) are
\[
X_p = e^{-tJ}\varepsilon Z_p \qquad X_{pp} = e^{-tJ}\varepsilon Z_{pp}.
\]
The time derivative has a few more terms:
\begin{align*}
  X_t =e^{-tJ}\left\{
  -2 R_\theta \ve1
  - R \ve2
  + \varepsilon' Z
  - \varepsilon J Z + \varepsilon Z_t \right\}
\end{align*}
Express \(W = \langle X_t - X_{pp}/\|X_p\|^2, JX_p\rangle\) in terms of \(Z\), keeping in mind that \(\varepsilon=R^{-1}\),
\begin{multline}\label{eq-2-X} 
  W = \left\langle \varepsilon^2 Z_t - \frac{Z_{pp}}{\|Z_p\|^2}, JZ_p\right\rangle 
  - \langle\ve1, Z_p\rangle \\
  + 2\varepsilon R_\theta \langle \ve2, Z_p\rangle + \varepsilon\varepsilon'\langle Z, JZ_p\rangle - \varepsilon^2\langle Z, Z_p\rangle .
\end{multline}
We look for a cap in the form of a normal perturbation of the Grim Reaper curve, i.e.~we assume
\begin{equation}\label{eq-perturbed-GR-ansatz}
Z(t, p) = G(p) + f(t, p) JG_p(p)
\end{equation}
where
\[
G(p) = -\arcsin(\tanh p)\ve1 - \ln (\cosh p) \ve2
\]
is the arclength parametrization of the Grim Reaper.

Since \(G\) is an arclength parametrization \(G_p, JG_p\) are unit tangent and normal to the Grim Reaper.  Specifically,
\[
G_p = -\frac{1}{\cosh p}\ve1 -\tanh p \ve2.
\]
The parametrization \(G(p)\) traces the Grim Reaper out from right to left.  Furthermore, the curvature vector of the Grim Reaper is
\[
G_{pp}(p) = \kappa(p) JG_p, \text{ where }\kappa(p) = \frac{1}{\cosh p}.
\]

\subsection{Detailed computation of \(W\) on the cap} \label{subsec-computation-of-W}
We have
\begin{align*}
  Z_t &= f_t JG_p\\
  Z_p&= G_p+f_pJG_p + fJG_{pp}
       = (1-\kappa f) G_p + f_p JG_p \\
  JZ_p&= - f_p G_p + (1-\kappa f) JG_p  \\
  \|Z_p\|^2 &= (1-\kappa f)^2 + f_p^2 = 1 -2\kappa f +\kappa^2 f^2 + f_p^2 \\
  Z_{pp} &= -(\kappa_p f + 2\kappa f_p)G_p 
           +\bigl(f_{pp} + \kappa - \kappa^2 f\bigr)JG_p 
\end{align*}
Substituting in \eqref{eq-2-X}, and using \(\kappa = 1/\cosh p\), we get
\begin{align*}
  \left\langle \varepsilon^2 Z_t, JZ_p \right\rangle
  &= \varepsilon^2 (1-\kappa f)f_t \\
  \Bigl\langle \frac{Z_{pp}}{\|Z_p\|^2}, JZ_p \Bigr\rangle
  &= \frac{(1-\kappa f)(f_{pp}+\kappa(1-\kappa f)) + \kappa_p ff_p +2\kappa f_p^2}
    {1-2\kappa f + \kappa^2 f^2 + f_p^2} \\
  \langle \ve1, Z_p \rangle &= -\kappa(1-\kappa f) + f_p \tanh p .
\end{align*}
Thus after substituting and expanding we find
\begin{align}\label{eq-W-all-terms}
  W= \varepsilon^2 (1-\kappa f)f_t 
  - \frac{\kappa + f_{pp} - 2\kappa^2 f-\kappa ff_{pp} + \kappa_p ff_p  + \kappa^3f^2 +2\kappa f_p^2}
  {1-2\kappa f + \kappa^2 f^2 + f_p^2}
  + \kappa(1-\kappa f) \\ - f_p \tanh p 
  - \varepsilon^2 \langle Z, Z_p \rangle
  + 2\varepsilon R_\theta \langle \ve2, Z_p \rangle 
  + \varepsilon\varepsilon' \langle  Z, JZ_p \rangle .\notag
\end{align}
We will choose \(f(t, p)\) so as to make \(W\) integrable in space and time.  To find \(f\) we linearize the expression for \(W\) and solve the resulting first order equation for \(f\).  It turns out that one solution is of the form \(f(t, p) = \tau^{-1}F(p)\) for a function \(F\) that is of polynomial growth for \(|p|\to\infty\).  We restrict our attention to the region
\begin{equation}
|p|\leqslant 2K\ln \tau,\qquad  \tau \gg 1,
\end{equation}
where \(K=100\) is a fixed, largish, constant.  We will assume that \(f\) and its derivatives are bounded by
\begin{equation}\label{eq-f-bounded}
|f|+|f_p|+|f_{pp}| + \tau |f_t| \lesssim \tau^{-1+\delta}
\text{ for }|p|\leqslant 2K\ln\tau.
\end{equation}
Here, and in what follows, when we write estimates for remainder terms of the form \(\cO(\tau^{-m+\delta})\), the estimate is implicitly meant to hold ``for all \(\delta>0\).''

The bound \eqref{eq-f-bounded} will certainly hold if \(f(t, p) = \tau^{-1}F(p)\) for some function \(F(p)\) for which \(F(p)\), \(F'(p)\), and \(F''(p)\) grow polynomially as \(p\to\pm\infty\).

We now consider the many terms in \eqref{eq-W-all-terms} that add up to \(W\).  To begin, we have for \(|p|\leqslant 2K\ln\tau\),
\[
\varepsilon^2 (1-\kappa f)f_t =\cO(\tau^{-2+\delta}),
\]
and also
\[
\frac{-\kappa ff_{pp} + \kappa_p ff_p + \kappa^3f^2 +2\kappa f_p^2} {1-2\kappa f + \kappa^2 f^2 + f_p^2} = \cO(\tau^{-2+\delta}) \text{ for any }\delta>0.
\]
It follows from \(Z=G+fJG_p\) that
\begin{align*}
  Z= G + \cO\bigl(\tau^{-1+\delta}\bigr),\qquad
  Z_p = G_p + \cO\bigl(\tau^{-1+\delta}\bigr).
\end{align*}
Hence, for \(|p|\leqslant 2K\ln\tau\),
\begin{align*}
  \|G(p)\|&\lesssim \ln\tau &
                              \varepsilon^2\langle Z, Z_p \rangle
  &= \tau^{-1}\langle G, G_p \rangle + \cO\bigl(\tau^{-2+\delta}\bigr) \\
  \varepsilon R_\theta \langle \ve2 , Z_p \rangle
  &=\tau^{-1}\langle \ve2, G_p \rangle + \cO\bigl(\tau^{-2+\delta}\bigr) 
  &\quad\varepsilon\varepsilon'\langle Z, JZ_p \rangle 
  &= \cO\bigl(\tau^{-2+\delta}\bigr).
\end{align*}
In this computation we have used the expansions \(\varepsilon = R(t)^{-1} = \cR(-2t)^{-1} = \tau^{-1/2} + \cO(\tau^{-3/2})\) and \(R_\theta = \cR'(-2t) = \tau^{-1/2} + \cO(\tau^{-3/2})\) that follow from the expansions of \(\cR(\theta)\) and \(\cR'(\theta)\) in appendix~\ref{sec-YY-properties}.

So far we have
\begin{align*}
  W = \cO\bigl(\tau^{-2+\delta}\bigr)
  - \frac{\kappa + f_{pp} - 2\kappa^2 f}
  {1-2\kappa f + \kappa^2 f^2 + f_p^2}
  + \kappa(1-\kappa f) - f_p \tanh p \\
  + \tau^{-1} \langle 2\ve2 - G, G_p \rangle
\end{align*}
We can simplify the fraction (for \(|p|\leqslant 2K\ln\tau\)) by using
\[
\frac{1}{1-2\kappa f + \kappa^2 f^2 + f_p^2} =1+2\kappa f - \kappa^2 f^2 - f_p^2+ \frac{(-2\kappa f + \kappa^2 f^2 + f_p^2)^2} {1-2\kappa f + \kappa^2 f^2 + f_p^2} 
=1+2\kappa f + \cO\bigl(\tau^{-2+\delta}\bigr)
\]
which implies
\begin{align*}
  \frac{\kappa + f_{pp} - 2\kappa^2 f}
  {1-2\kappa f + \kappa^2 f^2 + f_p^2}
  = \kappa + 2\kappa^2 f + f_{pp} -2\kappa^2 f+ \cO\bigl(\tau^{-2+\delta}\bigr) 
  = \kappa + f_{pp} + \cO\bigl(\tau^{-2+\delta}\bigr) ,
\end{align*}
and hence
\[
W = -f_{pp}-\tanh(p) f_p - \kappa^2 f +\frac{1}{\tau} \langle 2\ve2-G, G_p \rangle +\cO(\tau^{-2+\delta}).
\]
\subsection{Computation of the correction term} \label{subsec-computation-of-correction}
We look to perturb the Grim Reaper with a term of the form
\[
f(t, p) = \frac{F(p)}{\tau},
\]
where \(F\) is a solution of
\[
\cL F \isdef  F_{pp}+\tanh(p) F_p + \kappa^2 F = \langle 2\ve2-G, G_p \rangle.
\]
The linear operator can be factored
\[
\cL = \frac{d^2}{dp^2}+\tanh(p)\frac{d}{dp} + \frac{1}{\cosh^2(p)} =\frac{d}{dp} \circ \frac{1}{\cosh p}\circ \frac{d}{dp} \circ \cosh(p)
\]
while we also have
\[
\langle 2\ve2-G, G_p \rangle = \frac{d}{dp}\bigl\langle 2\ve2-\tfrac12 G, G \bigr\rangle,
\]
all of which allows us to solve the equation for \(F\):
\begin{equation}
\label{eq-F-general-solution}
F(p) = \frac{A}{\cosh p} + B\tanh p + \int_0^p \frac{\cosh r}{\cosh p} \bigl\langle 2\ve2-\tfrac12 G(r), G(r) \bigr\rangle \, dr.
\end{equation}
It appears that the first term is of no use, so we set \(A=0\).  The resulting function \(F(p)\) is an odd function of \(p\).  We use the asymptotic behavior of \(\langle 2\ve2-\tfrac12 G, G \rangle\) for large \(p\) to find an expansion for the integral as \(p\to+\infty\).

Consider
\[
I \isdef \int_0^p\frac{\cosh r}{\cosh p} \bigl\langle 2\ve2-\tfrac12 G(r), G(r) \bigr\rangle \, dr.
\]
The explicit expression for \(G\) implies
\begin{align*}
  \langle \ve2, G(r) \rangle 
  &= -\ln\cosh r \\
  \|G(r)\|^2
  &= \bigl(\ln\cosh r\bigr)^2 + \bigl(\arcsin\tanh r\bigr)^2  
    =\bigl(\ln\cosh r\bigr)^2 + \frac{\pi^2}{4} + \cO(e^{-r})  
    \qquad (r\to\infty).  
\end{align*}
To compute \(I\) we substitute \(\lambda=\ln\cosh p\), \(\mu=\ln\cosh r\), which leads to
\[
I = -\int_0^\lambda e^{\mu-\lambda} \bigl\{2\mu + \tfrac12 \mu^2 + \tfrac{\pi^2}8 + \cO(e^{-\mu})\bigr\} \frac{d\mu}{\sqrt{1-e^{-2\mu}}}
\]
The integrand is singular but integrable at \(\mu=0\).  To deal with this singularity split
\[
(1-e^{-2\mu})^{-1/2} = 1 + \bigl[(1-e^{-2\mu})^{-1/2} - 1\bigr], \text{ with } 0\leqslant (1-e^{-2\mu})^{-1/2} - 1 \lesssim \frac{1}{\sqrt\mu} e^{-2\mu}.
\]
Replacing \((1-e^{-2\mu})^{-1/2}\) by \(1\) therefore introduces an extra term that is bounded by \(\cO(e^{-\lambda})\).  Hence we have
\begin{align*}
  I&=\cO(e^{-\lambda}) - \int_0^\lambda e^{\mu-\lambda}\bigl\{2\mu + \tfrac12 \mu^2 + \tfrac{\pi^2}8 + \cO(e^{-\mu})\bigr\} \, d\mu \\
   &=\cO(\lambda^3 e^{-\lambda}) - \int_0^\lambda e^{\mu-\lambda}\bigl\{2\mu + \tfrac12 \mu^2 + \tfrac{\pi^2}8 \bigr\} \, d\mu \\
   &=\cO(\lambda^3 e^{-\lambda}) -\bigl\{2\lambda + \tfrac12 \lambda^2 + \tfrac{\pi^2}8\bigr\}
     +\bigl\{2+\lambda\bigr\}-\{1\}\\
   &=-\tfrac12 \lambda^2 - \lambda - \tfrac{\pi^2}{8}+ 1 + \cO(\lambda^3  e^{-\lambda})
\end{align*}
Thus, for large \(p\) we get
\[
F(p) = B -\tfrac12 (\ln\cosh p)^2 - \ln\cosh p - \tfrac{\pi^2}{8}+ 1 +\cO\bigl(|p|^3e^{-|p|}\bigr) \qquad (p\to+\infty).
\]
Since \(F(p)\) is an odd function we also have
\[
F(p) = -\left\{B -\tfrac12 (\ln\cosh p)^2 - \ln\cosh p - \tfrac{\pi^2}{8}+1 \right\} +\cO\bigl(|p|^3e^{-|p|}\bigr) \qquad (p\to-\infty).
\]

Applying this to \(Z= G + \tau^{-1}F(p)JG_p\) we get for the two components \(u\) and \(v\) of \(Z\) as \(p\to\pm\infty\):
\begin{align*}
  u(t, p)  &=-\arcsin \tanh p + \frac{1}{\tau}F(p) \tanh p
             =\mp\frac\pi2 + \frac{1}{\tau}F(p) \tanh p + \cO(e^{-|p|})\\
  v(t, p)  &= -\ln\cosh p -\frac{1}{\tau}\frac{F(p)}{\cosh p} 
             =- \ln\cosh p  + \cO(|p|^3e^{-|p|})
\end{align*}
We can again eliminate \(p\) when \(p\) is large by using
\[
\ln \cosh p = -v  +\cO(p^2 e^{-|p|}), \qquad (p\to\pm\infty)
\]
which leads to
\begin{multline}
\label{eq-utp-expansion-1}
u(t, p) = \mp\frac\pi2
-\frac{v^2 - 2v +\pi^2/4 + 2(1 + B)}{2\tau}
+\cO(|p|^3 e^{-|p|}) \qquad(p\to\pm\infty)
\end{multline}
We now determine \(B\) by matching \eqref{eq-utp-expansion-1} with the representation of the Yin-Yang arms in \((u,v)\) coordinates that we found in \eqref{eq-leaves-in-uv-expansion}.  Setting \(y=\mp \pi/2\) in \eqref{eq-leaves-in-uv-expansion} we find for the outer and inner Yin-Yang arms
\begin{equation}
\label{eq-outer-leaves-expansion}
u = \mp\frac \pi2 - \frac{\pi^2/4+v^2-2v}{2\tau} + \cO(\tau^{-2+\delta})
\end{equation}
If \(|p|\geqslant \frac K2 \ln\tau\) then \(|p|^3e^{-|p|} = \cO(\tau^{-K/2+\delta}) = o(\tau^{-2+\delta})\), and therefore the two expansions \eqref{eq-utp-expansion-1},\eqref{eq-outer-leaves-expansion} match if
\begin{equation}
 B =-1.
\end{equation}

To summarize, we choose the cap to be given by
\begin{equation}\label{eq-cap-ansatz}
X(t, p) = e^{-tJ}\bigl\{ R(t)\ve1 + \frac{1}{R(t)}Z(t, p)\bigr\}
\end{equation}
with
\begin{equation}\label{eq-cap-with-correction}
Z(t, p) = G(p) + \frac{F(p)}\tau JG_p(p)
\end{equation}
and, from \eqref{eq-F-general-solution},
\begin{equation}\label{eq-explicit-correction}
  F(p) =\, -\tanh p + \int_0^p \frac{\cosh r}{\cosh p} \bigl\langle 2\ve2-\tfrac12 G(r), G(r) \bigr\rangle \, dr
\end{equation}

\subsection{Definition of the smooth interpolation of cap and arms} \label{subsec-interpolation-region}
In the \((u,v)\) coordinates, according to \eqref{eq-leaves-in-uv-expansion}, the Yin-Yang arms are given by
\[
u = U_\pm(t, v) \text{ for } |v|\leqslant 3K\ln \tau
\]
where we abbreviate \(U_\pm(t, v) = U_{\pm\pi/2}(t, v)\).  In the same \((u,v)\) coordinates the ends of the cap are given by \eqref{eq-utp-expansion-1} with, \(B=-1\), i.e.
\[
u = h_\pm(t, v) \text{ for } \tfrac12K \ln\tau \leqslant -v\leqslant 2K\ln \tau.
\]
We constructed the cap so that both \(U_\pm\) and \(h_\pm\) have the same asymptotic behavior, namely,
\begin{equation}\label{eq-hpm-def}
U_\pm(t, v) , h_\pm(t, v) =  
\pm \frac\pi2 -\frac1\tau \Bigl\{\tfrac12 v^2 - v + \pi^2/8\Bigr\}
+ \cO(\tau^{-2+\delta}).
\end{equation}

Choose a smooth nondecreasing function \(\eta:\R\to\R\) with \(\eta(u)=0\) for \(u\leqslant \frac 12\) and \(\eta(u) = 1\) for \(u\geqslant 2\), and define
\begin{align}\label{eq-interpolation}
  k_\pm(t, v)
  &= \eta\Bigl(\frac{-v}{\ln\tau}\Bigr) U_\pm(t, v) +
    \left\{1-\eta\Bigl(\frac{-v}{\ln\tau}\Bigr) \right\} h_\pm(t, v)\nonumber \\
  &= U_\pm(t, v) + \eta\Bigl(\frac{-v}{\ln\tau}\Bigr)
    \bigl(U_\pm(t, v)- h_\pm(t, v)\bigr).
\end{align}
The graphs of these two functions are \(Z\)-coordinate representations of curve segments that smoothly interpolate between the two ends of the cap and the two Yin-Yang arms.  The two segments are parametrized by
\[
X_\pm(t, v) = e^{-tJ}\bigl\{ R\ve1 + R^{-1}Z_\pm(t, v)\bigr\} \text{ with } Z_\pm(t, v) = k_\pm(t, v)\ve1 + v\ve2.
\]
It follows from~\eqref{eq-2-X} that the Curve Shortening Deficit for such curves is given by \((\kappa-V)ds = W_v dv\) with
\begin{equation}\label{eq-error-on-interpolation}
W_v[k] = -R^{-2} k_t + \frac{k_{vv}}{1+k_v^2} - k_v +2\frac{R_\theta}{R} + 2\frac{R_\theta}{R^3}(vk_v - k) -R^{-2}(v+kk_v).
\end{equation}

\subsection{Derivative bounds for \(U_\pm\), \(h_\pm\), \(k_\pm\)}
Careful scrutiny of the construction of \(U_\pm(t, v)\) and \(h_\pm(t, v)\) shows that the remainder terms \(\cO(\tau^{-2+\delta})\) in \eqref{eq-hpm-def} may be differentiated.  This implies that the functions \(U_\pm\) and \(h_\pm\) satisfy
\begin{equation} \label{eq-ghk-derivative-estimate} |k_{vv}|+|k_v|+|k_t| \lesssim \frac{\ln\tau}{\tau} \lesssim \tau^{-1+\delta}
\end{equation}
for \(\frac12\ln\tau \leqslant -v\leqslant 2K\ln\tau\), and large enough \(\tau\).

The derivatives of the gluing function \(\eta(v/\ln\tau)\) are
\[
\eta_v = -\frac{\eta'(-v/\ln\tau)}{\ln\tau},\quad \eta_{vv} = \frac{\eta''(-v/\ln\tau)}{(\ln\tau)^2},\quad \eta_t = -\frac{4v\eta'(-v/\ln\tau)}{(\ln\tau)^2},
\]
so they are bounded by
\[
|\eta_v|+|\eta_{vv}|+|\eta_t| \lesssim \bigl(\ln\tau\bigr)^{-1}
\]
for \(\frac12\ln\tau \leqslant -v \leqslant 2K\ln\tau\), and large enough \(\tau\).

It follows that the interpolating functions \(k_\pm = \eta U_\pm + (1-\eta) h_\pm \) also satisfy \eqref{eq-ghk-derivative-estimate}.

\subsection{Estimating \(W_v[k_\pm]\)}
We show that
\begin{equation}
\label{eq-Wq-kpm-bound}
\big|W_v[k_\pm] \big|\lesssim \tau^{-2+\delta}
\end{equation}
holds in the region \(\frac12 K\ln\tau \leqslant -v\leqslant 2K\ln\tau\), for sufficiently large \(\tau\).

If \(k\) is any of the functions \(U_\pm, h_\pm, k_\pm\) then we have
\begin{align*}
  \frac{|k_v^2k_{vv}|}{1+k_v^2} \leqslant |k_v^2k_{vv}| \lesssim \tau^{-3+\delta},\quad
  R^{-2}|kk_v| \lesssim \tau^{-3+\delta},\text{ and }
  -R^{-2}|k_t|\lesssim \tau^{-2+\delta}.
\end{align*}
Furthermore \(|k|\lesssim 1\), \(|vk_v|\lesssim \tau^{-1+\delta}\), and \(R_\theta\sim R^{-1}\lesssim \tau^{-1/2}\) lead to
\[
\frac{R_\theta}{R^3} |vk_v-k| \lesssim \tau^{-2}.
\]
Hence
\[
W_v[k] = k_{vv} -k_v +2\frac{R_\theta}{R} - v R^{-2} +\cO(\tau^{-2+\delta})
\]
holds for all six functions \( k \in \{U_\pm, h_\pm, k_\pm\}\).

To simplify our notation we drop the \(\pm\) subscript for now and expand the derivatives of \(k=\eta U+ (1-\eta)h\) (with \(U=U_\pm\) and \(h=h_\pm\)),
\begin{gather*}
k_v = \eta U_v + (1-\eta)h_v+ \eta_v (U-h),\\
k_{vv} = \eta U_{vv} + (1-\eta)h_{vv}+ 2\eta_v (U-h)_v + \eta_{vv}(U-h).
\end{gather*}
Since we have matched the two cap ends with the Yin-Yang arms, it follows from \eqref{eq-hpm-def} that the difference \(U-h\) and its derivatives are bounded by
\[
|U-h|+|(U-h)_v|+|(U-h)_{vv}|\lesssim \tau^{-2+\delta}.
\]
Together these inequalities give us the desired estimate for \(W_v[k_\pm]\), namely
\begin{align}\label{eq-deficit-bound-interpolation}
  W_v[k]
  &= W_v[\eta U+(1-\eta)h] \nonumber \\
  &=  \bigl(\eta U+(1-\eta)h\bigr)_{vv}
    -\bigl(\eta U+(1-\eta)h\bigr)_v
    +2\frac{R_\theta}{R} - v R^{-2} +\cO(\tau^{-2+\delta}) \nonumber \\
  &=\eta W_v[U]+(1-\eta)W_v[h] + \cO(\tau^{-2+\delta}) \nonumber \\
  &=\cO(\tau^{-2+\delta}).
\end{align}

\begin{definition}[The Approximate Solution]\label{def-approximate-solution}
Let $\cC_*(t): (-\infty, -T) \times \R  \rightarrow \R^2$ for some sufficiently large $T >0$ be the family of smooth curves formed by the concatenation of the Yin-Yang leaves $Y(\theta, t, -\tfrac{\pi}{2})$, $Y(\theta, t, +\tfrac{\pi}{2})$ cut off in a neighborhood of $R(t)\vE1(-t)$ and glued to the cap $X(t,p)$ defined by the ansatzes \eqref{eq-cap} and \eqref{eq-perturbed-GR-ansatz}, with $f(t,p) = \tau^{-1}F(p)$ and $F(p)$ given by \eqref{eq-F-general-solution} with $A=0$ and $B=-1$. The gluing between the arms of the cap and the two Yin-Yang segments is given by the interpolation $k_{\pm}(t,v)$ in \eqref{eq-interpolation}, which is done in a neighborhood of $R(t)\vE1(-t)$.
\end{definition}

\begin{lem}\label{lem-error-bound} The error 
\[
\cE(T) = \int_{-\infty}^{-T}\int_{\cC_*(t)}  \left| V -\kappa\right|\, ds\, dt <\infty 
\]
is finite on the approximate solution $\cC_*(t): (- \infty, -T) \times \R \rightarrow \R^2$ given by Definition \ref{def-approximate-solution}.
\end{lem}

\begin{proof}
It suffices to show that the Curve Shortening Deficit $|V- \kappa| ds$ is $L^1$-integrable in space and time on three regions: the cap, the transition region discussed in \S \ref{subsec-interpolation-region}, and the unmodified Yin-Yang curve. Since the Yin-Yang curve is a solution to Curve Shortening Flow, the deficit $|V- \kappa|ds = 0$, which leaves only cap and the transition region as contributing to the error.

The cap, given by the expressions \eqref{eq-cap-ansatz}, \eqref{eq-cap-with-correction}, and \eqref{eq-explicit-correction}, is defined on the region $|p| \le 2K\ln \tau$, where $p$ is the arc-length coordinate $p$ for the grim reaper $G(p)$. In \S \ref{subsec-computation-of-W}-\S \ref{subsec-computation-of-correction}, the Curve Shortening Deficit $|V- \kappa| ds$ is written in terms of $p$ as $ W dp$, and it is shown to be $W = \cO(\tau^{-2+\delta})$ for $\delta > 0$. Integrating over the cap, we have
\[
\int_{\text{cap}} |V- \kappa| ds \le \int_{-2K\ln \tau}^{2K\ln \tau} W dp \le 4K\ln \tau \cO(\tau^{-2+\delta}) = \cO(\tau^{-2+\delta}).
\]
This quantity is integrable in time, and thus the contribution to the error $\cE$ on the cap is bounded.

In \S \ref{subsec-interpolation-region}, the Curve Shortening Deficit on the transition region is written in terms of the parameter $v$ as $W_v[k]dv$ on the interval $\tfrac{1}{2} \le -v \le 2K\ln \tau$. Furthermore, in \eqref{eq-deficit-bound-interpolation} it is shown that $W_v[k] = \cO(\tau^{-2+\delta})$. Integrating over both curves in the transition region, we have 
\[
\int_{\text{trns. reg.}} |V- \kappa| ds = 2 \int_{\tfrac{1}{2}K\ln \tau}^{2K\ln \tau} W_v[k] dv \le 3K\ln \tau \cO(\tau^{-2+\delta}) = \cO(\tau^{-2+\delta}).
\]
As before, this is integrable in time and thus the contribution to the error is bounded. Therefore, the sum of the integrals of the deficit over each region is $\cO(\tau^{-2+\delta})$. This completes the proof of the lemma.
\end{proof}

\section{Area decreasing property of Space Curve Shortening}
\label{sec-area-growth-bound}
In 1991 Altschuler and Grayson \cite{MR1158337} observed that for two solutions of space curve shortening the area of the minimal surface spanning them is non increasing.  Here we elaborate on this and prove a similar result without using the existence of the minimal surface.

\subsection{Moving space curves}
\label{sec-moving-space-curves}
For an immersed curve \(X:\R\to\R^n\) one defines the arc length one-form \(ds\) and the arc length derivative \(\partial_s\) of any quantity \(f:\R\to\R\) by
\[
ds = \|X_p\|\, dp, \text{ and } \frac{\partial f} {\partial s} = \frac{1} {\|X_p\|} \frac{\partial f} {\partial p}.
\]
The unit tangent and curvature of the curve are \(X_s = \frac{X_p} {\|X_p\|}\) and \(X_{ss}\).

A moving family of space curves is a map \(X:(t_0, t_1)\times\R \to \R^n\).  The family evolves by Curve Shortening if it satisfies \( X_t^\perp = X_{ss} \), i.e., if for some smooth function \(\lambda(t, p)\) one has
\begin{equation}
\label{eq-scs}
X_t = X_{ss}+ \lambda X_s 
= \frac{1}{\|X_p\|}\frac{\partial}{\partial p} \Bigl(\frac{X_p}{\|X_p\|}\Bigr) + \lambda \frac{X_p}{\|X_p\|}
\end{equation}
Since \(X_s\perp X_{ss}\), one can always find \(\lambda\) from \(\lambda = \langle X_{t}, X_s \rangle\).

\subsection{Evolution of arc length and the commutator \texorpdfstring{\([\partial_t, \partial_s]\)}{[dt, ds]}}
\label{sec-evolution}
The following are commonly used relations.  We record them here for completeness, and also because we allow the velocity \(X_t\) of the parametrizations to have a nonvanishing tangential component.  Assuming that \(X_t=X_{ss}+\lambda X_s\) one has
\begin{equation}
\label{eq-ds-evol}
\partial_t \|X_p\| = (\lambda_s - \kappa^2)\|X_p\|,
\qquad
\frac{\partial }{\partial t}ds  = (\lambda_s-\kappa^2) ds = d\lambda - \kappa^2 ds,
\end{equation}
and
\begin{equation}
\label{eq-ds-dt-commutator}
[\partial_t, \partial_s] = (-\lambda_s+\kappa^2)\partial_s .
\end{equation}
\begin{proof}
   
We have
\[
\frac{\partial}{\partial t}\|X_p\| = \Bigl\langle \frac{X_p}{\|X_p\|}, X_{tp}\Bigr\rangle = \langle X_s, X_{tp}\rangle = \langle X_s, X_{ts} \rangle \|X_p\|,
\]
and hence
\[
\frac{\partial}{\partial t} ds = \frac{\partial}{\partial t} \bigl(\|X_p\| dp\bigr) = \langle X_s, X_{tp}\rangle dp = \bigl\langle X_s, X_{ts}\bigr\rangle ds .
\]
The evolution equation \(X_t=X_{ss}+\lambda X_s\) then implies
\[
\bigl\langle X_s, X_{ts}\bigr\rangle 
= \langle X_s, X_t \rangle_s - \bigl\langle X_{ss}, X_t\bigr\rangle 
= \lambda_s - \bigl\langle X_{ss}, X_{ss}+\lambda X_s\bigr\rangle 
=\lambda_s-\|X_{ss}\|^2 =\lambda_s- \kappa^2,
\]
which directly implies \eqref{eq-ds-evol}.  Using \eqref{eq-ds-evol} we get
\[ [\partial_t, \partial_s] = \left[\partial_t, \|X_p\|^{-1}\partial_p\right] = -\frac{\partial_t\|X_p\|}{\|X_p\|^{2}} \partial_p =-(\lambda_s-\kappa^2)\partial_s.\qedhere
\]
\end{proof}

\subsection{Dependence on a parameter}
Let \([\epsilon_0, \epsilon_1]\subset\R\) be some parameter interval, and let \(X:[\epsilon_0, \epsilon_1]\times(t_0,t_1)\times\R \to \R^n\) be a family of moving curves that depends on a parameter \(\epsilon\in [\epsilon_0, \epsilon_1]\).  We compute the evolution of the first variation
\[
X_\epsilon = \partial_\epsilon X(\epsilon, t, p).
\]
Throughout the computation we will assume that the parametrization \(X\) is such that
\begin{equation}
\label{eq-hypothesis-Xeps-perp}
X_\epsilon \perp X_p(\epsilon, t, p) \text{ for all \((\epsilon, t, p)\)}
\end{equation}
For any given parametrization \(\tilde X\) one can find a reparametrization \(\varphi(\epsilon, t, p)\) so that \(X(\epsilon, t, p) = \tilde X(\epsilon, t, \varphi(\epsilon, t, p))\) satisfies \eqref{eq-hypothesis-Xeps-perp}.

If \(X:[\epsilon_0,\epsilon_1]\times[0,1]\to\R^n\) is injective with \(X_\varepsilon\perp X_p\), then the double integral
\[
\int_{\epsilon_0}^{\epsilon_1} \int_{p=0}^1 \|X_\epsilon\| \, ds \,d\epsilon \qquad (\text{where }ds = \|X_p\|dp)
\]
is the area (2-dim Hausdorff measure) of the image \(X([\epsilon_0,\epsilon_1]\times[0,1])\).  If \(X\) merely satisfies \(X_\epsilon\perp X_p\), without necessarily being injective, then the area formula implies that the double integral is bounded from below, by
\begin{equation}
\ell(X)\isdef
\int_{\epsilon_0}^{\epsilon_1} \int_{p=0}^1 \|X_\epsilon\| \, ds \,d\epsilon 
\geqslant
\mathcal{H}^2(X([\epsilon_0,\epsilon_1]\times[0,1]))
\end{equation}
We will call the integral \(\ell(X)\) the \emph{length of the homotopy \(X\)}, and we will show that Curve Shortening decreases the length of homotopies.

The following improvement of the inequality \(\|X_\epsilon\|_s \leqslant \|X_{\epsilon s}\|\) (which follows from the Cauchy-Schwarz inequality) will be useful.

\subsection{Lemma}\label{lem-improved-triangle-ineq}\itshape 
Assuming \eqref{eq-hypothesis-Xeps-perp} we have
\[
\bigl(\|X_\epsilon\|_s\bigr)^2 \leqslant \|X_{\epsilon s}\|^2 - \langle X_{ss}, X_\epsilon \rangle^2.
\]
\upshape
\begin{proof}
Split \(X_{\epsilon s}\) into tangential and orthogonal components:
\[
X_{\epsilon s} = P + \langle X_{\epsilon s}, X_s \rangle X_s
\]
Since \(X_\epsilon\perp X_s\) we have $\langle X_{\epsilon s}, X_s\rangle = \partial_s\langle X_\epsilon, X_s\rangle-\langle X_\epsilon, X_{ss}\rangle = -\langle X_\epsilon, X_{ss}\rangle$.  
Therefore
\[
X_{\epsilon s} = P - \langle X_\epsilon, X_{ss} \rangle X_s,
\]
and thus
\[
\|X_{\epsilon s}\|^2 = \|P\|^2 + \langle X_\epsilon, X_{ss} \rangle^2.
\]
On the other hand
\[
\|X_\epsilon\|_s = \Bigl\langle \frac{X_\epsilon}{\|X_\epsilon\|}, X_{\epsilon s} \Bigr\rangle = \Bigl\langle \frac{X_\epsilon}{\|X_\epsilon\|}, P \Bigr\rangle \leqslant \|P\|,
\]
where we have again used that \(X_s\perp X_\epsilon\).  Combining these observations we arrive at
\[
\bigl(\|X_\epsilon\|_s\bigr)^2 - \|X_{\epsilon s}\|^2 \leqslant \|P\|^2 - \|X_{\epsilon s}\|^2 = - \langle X_\epsilon, X_{ss} \rangle^2,
\]
as claimed.
\end{proof}

\subsection{The commutator \texorpdfstring{\([\partial_\epsilon, \partial_s]\)}{[deps, ds]}}\itshape
Assuming \eqref{eq-hypothesis-Xeps-perp} one has
\[ [\partial_\epsilon, \partial_s] = \langle X_{ss}, X_\epsilon \rangle\partial_s \text{ and } [\partial_\epsilon , \partial_s^2] = 2\langle X_{ss}, X_\epsilon \rangle \partial_s^2 + \langle X_{ss}, X_\epsilon \rangle_s \partial_s
\]
\upshape
\begin{proof}
The computation follows the same pattern as the derivation of \eqref{eq-ds-dt-commutator}.  Here we have no equation for \(X_\epsilon\), but we do know that \(X_\epsilon\perp X_s\).  Thus
\[
\partial_\epsilon \|X_p\| = \left\langle\frac{X_p}{\|X_p\|}, X_{p\epsilon}\right\rangle = \langle X_s, X_{\epsilon p} \rangle = \langle X_s , X_{\epsilon s} \rangle \|X_p\| = -\langle X_{ss}, X_\epsilon \rangle \|X_p\|.
\]
Apply this to \(\partial_s=\|X_p\|^{-1}\partial_p\) to get the commutator \([\partial_\epsilon, \partial_s]\).  The other commutator follows from expanding \([\partial_\epsilon, \partial_s^2] = [\partial_\epsilon, \partial_s]\partial_s + \partial_s[\partial_\epsilon, \partial_s]\).
\end{proof}
\subsection{Lemma}\label{lem-Xeps-diff-ineq}\itshape
The length of the first variation \(X_\epsilon\) satisfies the differential inequality
\[
\partial_t \|X_\epsilon\| -\lambda \partial_s \|X_\epsilon\| \leqslant \partial_s^2\|X_\epsilon\| + \kappa^2 \|X_\epsilon\|.
\]\upshape
\begin{proof}
Differentiating the evolution equation \eqref{eq-scs} for \(X\) we get
\begin{align*}
  \partial_t X_{\epsilon} = \partial_\epsilon X_{t}
  &= \partial_\epsilon\left(X_{ss}+\lambda X_s\right) \\
  &= X_{\epsilon ss} + 2 \langle X_{ss}, X_\epsilon \rangle X_{ss} 
    +\langle X_{ss}, X_\epsilon \rangle_s X_s
    +\lambda_\epsilon X_s + \lambda \partial_\epsilon X_{s}.
\end{align*}
Hence
\begin{align*}
  \partial_t X_{\epsilon} -\lambda \partial_s X_\epsilon 
  &= X_{\epsilon ss} + 2 \langle X_{ss}, X_\epsilon \rangle X_{ss} 
    +\langle X_{ss}, X_\epsilon \rangle_s X_s
    +\lambda_\epsilon X_s + \lambda [\partial_\epsilon, \partial_s]X\\
  &= X_{\epsilon ss} + 2 \langle X_{ss}, X_\epsilon \rangle X_{ss} 
    + \left\{\langle X_{ss}, X_\epsilon \rangle_s 
    +\lambda_\epsilon  + \lambda \langle X_{ss}, X_\epsilon \rangle\right\} X_s
\end{align*}
We next compute the evolution of \(\|X_\epsilon\|^2\), keeping in mind that \(X_\epsilon\perp X_s\):
\begin{align*}
  (\partial_t-\lambda \partial_s)\|X_\epsilon\|^2 &= 2\langle X_\epsilon, \partial_t X_{\epsilon} -\lambda \partial_s X_\epsilon \rangle \\
                             &= 2\langle X_\epsilon, X_{\epsilon ss} \rangle 
                               + 4 \langle X_{ss}, X_\epsilon \rangle^2 \\
                             &= \bigl(\|X_\epsilon\|^2\bigr)_{ss} - 2 \|X_{\epsilon s}\|^2 + 4\langle X_{ss}, X_\epsilon \rangle^2\\
                             &=2\|X_\epsilon\| \,\|X_\epsilon\|_{ss}
                               + 2 \bigl(\|X_\epsilon\|_s\bigr)^2 - 2 \|X_{\epsilon s}\|^2 + 4 \langle X_{ss}, X_\epsilon \rangle^2 .
\end{align*}
At this point we use Lemma~\ref{lem-improved-triangle-ineq}, to get
\[
(\partial_t-\lambda \partial_s)\|X_\epsilon\|^2 \leqslant 2\|X_\epsilon\| \,\|X_\epsilon\|_{ss} + 2 \langle X_{ss}, X_\epsilon \rangle^2.
\]
Since \((\partial_t-\lambda \partial_s) \|X_\epsilon\|^2 = 2\|X_\epsilon\|(\partial_t-\lambda \partial_s)\|X_\epsilon\|\), we have
\[
\partial_t\|X_\epsilon\| -\lambda \partial_s\|X_\epsilon\| \leqslant \|X_\epsilon\|_{ss} + \Bigl\langle X_{ss}, \frac{X_\epsilon}{\|X_\epsilon\|} \Bigr\rangle^2 \|X_\epsilon\| \leqslant \|X_\epsilon\|_{ss} + \kappa^2\|X_\epsilon\|.\qedhere
\]
\end{proof}

\subsection{Contractive property of Curve Shortening}
\label{sec-CS-contracts-homotopies}
If \(X^0, X^1:[t_0, t_1]\times[0, 1]\to\R^n\) are two solutions of Curve Shortening~\eqref{eq-scs}, then a homotopy \(\{X^\epsilon : 0\leqslant \epsilon\leqslant 1\}\) of solutions to Curve Shortening connecting them is, by definition, a map \(X:[0,1]\times[t_0, t_1]\times[0,1]\to\R^n\) such that \((t, p) \mapsto X(\varepsilon, t, p)\) is a solution of Curve Shortening with \(X^\epsilon(t, p) = X(\epsilon, t, p)\) for \(\epsilon\in\{0,1\}\).

Given any homotopy \(X^\epsilon\) between solutions \(X^0, X^1\) of Curve Shortening, one can always find a reparametrization \(\tilde X^\epsilon(t, p) = X^\epsilon(t, \varphi(\epsilon, t, p))\) for which \(\partial_\epsilon \tilde X^\epsilon \perp \partial_s \tilde X^\epsilon\) holds pointwise.  We will call such a homotopy a \emph{normal homotopy.}

Our main observation in this section is the following: \itshape if \(X^\epsilon\) is a normal homotopy between solutions \(X^0, X^1\) of Curve Shortening, then one has for each \(\epsilon\in[0,1]\) and \(t\in(t_0, t_1)\)
\begin{equation}
\label{eq-infinitesimal-contraction}
\frac{d}{dt}\int_{p=0}^1 \|\partial_\epsilon X (\epsilon, t, p)\|\,ds \leqslant 0
\end{equation}
and
\begin{equation}
\label{eq-homotopy-contraction}
\int_{\epsilon=0}^1\int_{p=0}^1 \|\partial_\epsilon X(\epsilon, t_1, p)\|\,ds\,d\epsilon
\leqslant
\int_{\epsilon=0}^1\int_{p=0}^1 \|\partial_\epsilon X(\epsilon, t_0, p)\|\,ds\,d\epsilon
\end{equation}\upshape
\begin{proof}
We use \eqref{eq-ds-evol} and Lemma \ref{lem-Xeps-diff-ineq} to differentiate under the integral:
\begin{multline*}
\frac{d}{dt}\int_{p=0}^1 \|X_\epsilon\|\,ds \leqslant\int\left\{\|X_\epsilon\|_{ss}+\kappa^2 \|X_\epsilon\| + \lambda\|X_\epsilon\|_s +(\lambda_s - \kappa^2)\|X_\epsilon\|
\right\}\,ds \\
=\int\bigl\{\|X_\epsilon\|_{s} + \lambda\|X_\epsilon\| \bigr\}_s\,ds =0.
\end{multline*}
This implies \eqref{eq-infinitesimal-contraction}.  Integration in \(\epsilon\) and in time then leads to \eqref{eq-homotopy-contraction}.
\end{proof}

\subsection{Deviation from an approximate solution}\label{sec-deviation-from-CSF} We
now consider the case of two moving curves \(X^0, X^1:[t_0, t_1]\times[0,1]\to\R^n\) with the same initial value, i.e.~with \(X^0(t_0, p) = X^1(t_0, p)\) for all \(p\in[0,1]\).  We assume that \(X^1\) is a solution of Curve Shortening but allow \(X^0\) to be a general moving curve.  We measure its deviation from Curve Shortening in terms of
\begin{equation}\label{eq-dev-from-CS-def}
\Delta \isdef
\int_{t_0}^{t_1} \int_{p=0}^1 \bigl\| (\partial_t X^0)^\perp - \partial_s^2X^0\bigr\| \,ds\,dt.
\end{equation}
In the case of plane curves \(X:[t_0,t_1]\times [0, 1]\to\R^2\), we have \(\partial_t(X^0)^\perp = V JX^0_s\), where \(V\) is the normal velocity of the curve \(X^0\).  Therefore the integrand in \eqref{eq-dev-from-CS-def} is 
\[
\bigl\| (\partial_t X^0)^\perp - \partial_s^2X^0\bigr\|\, ds = |V-\kappa|\,ds.
\]
The quantity \(\Delta\) therefore coincides with the ``error'' \(\cE\) defined in \eqref{eq-error-def}.

Assume that for each \(\epsilon\in[t_0, t_1]\) there is a smooth solution \((t, p)\mapsto X(\epsilon, t, p)\) of Curve Shortening that is defined for \(t\in [\epsilon, t_1]\), and that has initial value \(X(\epsilon, \epsilon, p) = X^0(\epsilon, p)\).  After reparametrizing we may assume that \(X_\epsilon\perp X_p\) holds point-wise.  Then the final values of these solutions, i.e.~the curves \(p\mapsto X(\epsilon, t_1, p)\) form a normal homotopy from \(X^0(t_1, \cdot)\) to \(X^1(t_1, \cdot)\).  We will now show that
\begin{equation}
\label{eq-deviation-from-CS}
\int_{\epsilon=t_0}^{t_1} \int_{p=0}^1 \| \partial_\epsilon X(\epsilon, t_1,p)\| \,ds\,d\epsilon \leqslant \Delta.
\end{equation}
\begin{proof} Our argument is a nonlinear version of the Variation of Constants Formula, or of Duhamel's principle.

For \(\epsilon\in[t_0, t_1]\) we consider
\[
E(\epsilon) \isdef \int_{\bar\epsilon=t_0}^\epsilon \int_0^1 \|\partial_\epsilon X(\bar\epsilon, t_1, p)\|\,ds\,d\bar\epsilon.
\]
Then
\begin{align*}
  E'(\epsilon) = \int_0^1 \|\partial_\epsilon X(\epsilon, t_1, p)\|\,ds
\end{align*}
The contraction property \eqref{eq-infinitesimal-contraction} implies
\[
\int_0^1 \|\partial_\epsilon X(\epsilon, t_1, p)\|\,ds \leqslant \int_0^1 \|\partial_\epsilon X(\epsilon, \epsilon, p)\|\,ds.
\]
We compute \(X_\epsilon(\epsilon,\epsilon,p)\) by differentiating the relation \(X(\epsilon,\epsilon,p) = X^0(\epsilon,p)\) with respect to \(\epsilon\):
\[
\bigl(\partial_t X^0\bigr)(\epsilon, p) =\frac{\partial X^0(\epsilon, p)}{\partial\epsilon} =\frac{\partial X(\epsilon, \epsilon, p)}{\partial\epsilon} =X_\epsilon(\epsilon,\epsilon,p) + X_t(\epsilon,\epsilon,p).
\]
Since \((t, p)\mapsto X(\epsilon,t,p)\) evolves by Curve Shortening, we have \(X_t(\epsilon,\epsilon,p) = X_{ss}(\epsilon,\epsilon,p)+\lambda X_s\).  By definition of \(X\) we have \(X(\epsilon, \epsilon, p)= X^0(\epsilon, p)\), so \(X_{ss}(\epsilon,\epsilon,p) = X^0_{ss}(\epsilon,p)\).

We have parameterized the homotopy \(X(\epsilon,t,p)\) so that \(X_\epsilon\perp X_p\), and therefore
\[
\partial_\epsilon X(\epsilon,\epsilon,p) = \bigl(\partial_t X^0(\epsilon, p)\bigr)^\perp - X^0_{ss}(\epsilon, p).
\]
Hence
\[
E'(\epsilon)\leqslant \int_0^1 \|\partial_\epsilon X(\epsilon, \epsilon, p)\|\,ds =\int_0^1 \left\|\bigl(\partial_t X^0(\epsilon, p)\bigr)^\perp - X^0_{ss}(\epsilon, p) \right\| \,ds
\]
Integrate over \(\epsilon\in[t_0, t_1]\) to recover \eqref{eq-deviation-from-CS}.
\end{proof}

\subsection{Application to plane Curve Shortening}\label{sec-CS-area-decrease}
Let \(X^0, X^1:[t_0, t_1]\times[0,1]\to\R^2\) be two moving curves that are embedded at all time.  Assume \(X^1\) evolves by Curve Shortening, and assume that initially \(X^1\) lies in the closed region enclosed by \(X^0\), i.e.~for all \(p\in[0,1]\) the point \(X^1(t_0, p)\) lies in the region enclosed by the simple curve \(p\mapsto X^0(t_0, p)\).

Assume furthermore that at each time \(t_*\in[t_0, t_1]\) the area enclosed by \(p\mapsto X^0(t_*, p)\) is at least \(2\pi(t_1-t_*)\).  By the Gage-Hamilton-Grayson theorem this guarantees that the solution to Curve Shortening starting at \(X^0(t_*, \cdot)\) exists until time \(t_1\).

We now consider two homotopies \(A\) and \(B\) of evolving curves.  The first is the homotopy defined in the proof in the previous section~\ref{sec-deviation-from-CSF}, i.e.~for each \(\epsilon\in[t_0, t_1]\) we consider the solution \((t, p)\mapsto X(\epsilon, t, p)\) to Curve Shortening defined for \(t\in[\epsilon, t_1]\) and starting from \(X(\epsilon,\epsilon,p) = X^0(\epsilon, p)\).  Our first homotopy is then the family of final curves \(A(\epsilon, p) = X(\epsilon, t_1, p)\) of these solutions.  In section~\ref{sec-deviation-from-CSF} we showed that the length of the homotopy \(A\) is bounded from above by
\[
\ell(A)
= \int_{t_0}^{t_1}\int_0^1 \|A_\epsilon(\epsilon, p)\|\,ds\,d\epsilon
\leqslant \int_{t_0}^{t_1}\int_0^1 \left\|(\partial_t X^0)^\perp - X^0_{ss}\right\| \,ds\,dt.
\]

The second homotopy is constructed by evolving a homotopy between the two initial curves \(p\mapsto X^j(t_0, p)\) (\(j=0,1\)).  Since \(X^1\) initially lies inside \(X^0\) we can choose the homotopy \((\epsilon, p)\mapsto \bar X(\epsilon, t_0, p)\) so that its length is exactly the area of the region between the two initial curves, and so that the curve \(p\mapsto \bar X(\epsilon, t_0, p)\) lies inside the curve \(p\mapsto \bar X(\epsilon', t_0, p)\) if \(0\leqslant \epsilon \leqslant \epsilon'\leqslant 1\).  Given this initial homotopy let \((t, p)\mapsto \bar X(\epsilon, t, p)\) be the solution to Curve Shortening starting at \(\bar X(\epsilon, t_0, p)\).  Since all initial curves enclose \(X^1(t_0, p)\) the corresponding solutions exist for \(t\in[t_0, t_1]\), and possibly longer.  Our second homotopy is now \(B(\epsilon, p) = \bar X(\epsilon, t_1, p)\).  As explained in section~\ref{sec-CS-contracts-homotopies}, the length of the homotopy \(B\) is bounded by the length of the initial homotopy \((\epsilon, p)\mapsto \bar X(\epsilon, t_0, p)\).  We had chosen this initial homotopy so that its length is exactly the area between the two curves \(p\mapsto X^0(t_0, p)\) and \(p\mapsto X^1(t_0, p)\).

By concatenating the two homotopies \(A\) and \(B\) we obtain a combined homotopy \(A\#B\) between \(p\mapsto X^0(t_1, p)\) and \(p\mapsto X^1(t_1, p)\).  The length of this homotopy is bounded by
\begin{multline*}
\ell(A\#B)= \ell(A)+ \ell(B)
\leqslant \text{Area between \(X^0(t_0,\cdot)\) and \(X^1(t_0, \cdot)\)}\\
+ \int_{t_0}^{t_1}\int_0^1 \left\|(\partial_t X^0)^\perp - X^0_{ss}\right\| \,ds\,dt
\end{multline*}
If \(\Omega^0(t)\) and \(\Omega^1(t)\) are the regions enclosed by \(X^0(t, \cdot)\) and \(X^1(t,\cdot)\), then the homotopy \(A\#B\) between the two curves at time \(t_1\) must pass through each point in the interior of the symmetric difference \(\Omega^0(t_1)\triangle \Omega^1(t_1)\), as one sees by considering the winding numbers of the curves in the homotopy around any point in \(\Omega^0(t_1)\triangle \Omega^1(t_1)\).  It follows that the area of \(\Omega^0(t_1)\triangle \Omega^1(t_1)\) is a lower bound for the length of the homotopy \(A\#B\), and thus we conclude that
\begin{multline}
\label{eq-homotopy-area-bound}
\text{Area of }\Omega^0(t_1)\triangle \Omega^1(t_1)\leqslant \\
\text{Area between \(X^0(t_0,\cdot)\) and \(X^1(t_0, \cdot)\)} + \int_{t_0}^{t_1}\int_0^1 \left\|(\partial_t X^0)^\perp - X^0_{ss}\right\| \,ds\,dt.
\end{multline}

\section{Convergence}
\label{sec-convergence}
In this section, we obtain uniform curvature bounds on a sequence of ``really old solutions'' \(\{\cC_j(t) \}\) and extract a subsequence of solutions that converges locally smoothly to an ancient solution of curve shortening flow (CSF).

\begin{thm}\label{unifcurvbounds}
There exists a \(T\) such that for any \(T^* > T\), the curvatures \(|\kappa_j|\) of the ``really old solutions'' \(\{\cC_j(t)\}_j\) are bounded independently of \(j\) on the interval \([-T^* - 1/4, -T^*]\).
\end{thm}

The strategy to obtain these bounds is as follows: a) decompose an element \(\cC_j(t)\) of this sequence into the union of several graph representations, b) use the \(L^1\) bound on the error to obtain \(L^\infty\) estimates for these graphs, and c) apply the standard estimates for divergence-form quasilinear parabolic equations to establish a uniform curvature bound.

\medskip

Let \(T^*>0\) be a large positive number, which may be increased as necessary throughout this section.  The obvious candidates for the sequence of ``really old solutions'' are the CSF solutions defined on $[-j,-T^*]$ starting at $\cC_*(-j)$ at time $-j$ ---~call these $\Gamma_j(t)$.  At any time \(t \in [-j, -T^*]\), the unsigned area enclosed by the curves \(\cC_*(t)\) and \(\Gamma_j(t)\) is bounded by the quantity
\[ \int_{-j}^t \int_{\cC_*(\tau)} |V - \kappa| ds \; d\tau.
\]
By Lemma \ref{lem-error-bound}, this quantity is in \(L^1\) and given \(\epsilon >0\), we can find \(T^*>0\) such that
\[\cE(T^*) = \int_{-\infty}^{-T^*} \int_{\cC_*(\tau)} |V - \kappa| ds \; d\tau < \epsilon.
\]
By \eqref{eq-homotopy-area-bound}, this estimate gives a uniform bound on the unsigned area between \(\Gamma_j(t)\) and \(\cC_*(t)\) for any sufficiently large \(j\) and \(t < -T^*\), whenever defined.

\medskip

For simplicity of calculation, we consider an alternative sequence of ``really old solutions'' \(\{\cC_j(t)\}\), called ``square-profile approximations'', which do \textbf{not} satisfy the initial condition \(\cC_j(-j) = \cC_*(-j)\).  They are constructed as follows: Let \([\vartheta_*^-(t), \vartheta_*^+(t)]\) be the interval of polar angles that contain the approximate solution at a given time~\(t\).  Let \(r=\cR(\theta)\) be the central branch of the Yin-Yang foliation, so the Yin-Yang solution is given by the two branches \(r=\cR(\theta + \pi/2 -t)\) and \(r=\cR(\theta - \pi/2 -t)\).  We define \(\cC_j(t)\) to be the solution of curve shortening which at time \(t=-j\) is given by
\begin{itemize}
\item Two arms of the Yin-Yang soliton \(r=\cR(\theta - \pi/2 -t)= \cR(\theta - \pi/2 + j)\), and
  \(r=\cR(\theta + \pi/2 -t) = \cR(\theta + \pi/2 + j)\) truncated at \(\theta=\vartheta_*^+(-j)\);
\item a straight line segment connecting the two arms of the Yin-Yang soliton.  This segment is part of the ray \(\theta = \vt_*^+(-j)\).
\end{itemize}
Notice that at each time $t$ along the flow, the square-profile approximations \(\cC_j(t)\) enclose the curves $\Gamma_j(t)$, the CSF solutions starting from \(\cC_*(-j)\), and
that the area bounded by these two solutions stays constant along the flow, for all \(t \in [-j, -T^*]\).  The area \(A_j\) between \(\cC_j(-j)\) and \(\cC_*(-j)\) is small and goes to zero as \(j\to\infty\).  Thus, the area between the old solution \(\cC_j(t)\) with the ``square initial data'' and the approximate solution \(\cC_*(t)\) is bounded by
\[
  \mathop{\textrm{Area}} (\cC_j(t), \cC_*(t)) \leqslant A_j + \int_{-j}^t \int_{\cC_*(\tau)}
  |V-\kappa|\,ds\,d\tau.
\]

In order to improve these area bounds to \(L^\infty\) bounds, we will use the geometry of the \(\cC_j(t)\) and several properties of CSF.  In particular, we often appeal to the maximum principle and the following Sturmian property for intersections of curve shortening flows.

\begin{thm}\label{sturmian} Consider two CSF solutions
  \(\gamma_0, \gamma_1: [T_1, T_2) \times [0,1] \rightarrow \R^2\), for which
  \[ \partial \gamma_0 (t) \cap \gamma_1(t) = \partial \gamma_1(t) \cap \gamma_0(t) =
    \emptyset
  \]
  holds for any \(t \in [T_1, T_2)\).  Then the number of intersections of \(\gamma_0(t)\)
  and \(\gamma_1(t)\) is a finite and non-increasing function of \(t \in (T_1, T_2)\).  It
  decreases whenever \(\gamma_0\) and \(\gamma_1\) have a tangency.
\end{thm}

There is a useful related theorem for inflections points.

\begin{thm}\label{thm-inflections} 
  Let \(\gamma: [T_1, T_2) \times S^1 \rightarrow \R^2\) be a solution of CSF.
  Then, for any \(t \in (T_1, T_2)\), \(\gamma(t)\) has at most a finite number of
  inflection points, and this number does not increase with time.  In fact, it
  drops whenever the curvature \(\kappa\) has a multiple zero.
\end{thm}

While the curves \(\cC_j(t)\) are not convex, we do have a one sided curvature bound.
\begin{thm}\label{thm-kappa+xxs-Nonnegative}
If \(\kappa\) is the curvature of a counterclockwise oriented parametrization \(X\) of the curves \(\cC_j(t)\), then
\[
\kappa - \langle X, X_s\rangle >0.
\]
\end{thm}
\begin{proof}
Assuming that the parametrization \(X\) is normal (\(X_t\perp X_s\)), the curvature evolves by
\[
\kappa_t = \kappa_{ss}+\kappa^3.
\]
A short computation using \(X_t=X_{ss}\) and \(\|X_s\|=1\) shows that
\[
\bigl(\partial_t-\partial_{s}^2\bigr)\|X\|^2 = 2\langle X_t, X\rangle - 2 \langle X_{ss}, X\rangle - 2\|X_s\|^2 = -2.
\]
Differentiating with respect to arclength, using the commutator \([\partial_t, \partial_s] = \kappa^2\partial_s\), and also \(\partial_s \|X\|^2 = 2\langle X, X_s\rangle\) we get
\[
\partial_t \langle X, X_s\rangle
= \partial_s^2\langle X, X_s\rangle +  \kappa^2\langle X, X_s\rangle.
\]
Hence \(\kappa\) and \(\langle X, X_s\rangle\) satisfy the same linear equation.  Therefore \(\sigma = \kappa-\langle X, X_s\rangle\) also satisfies
\[
\sigma_t = \sigma_{ss}+\kappa^2 \sigma.
\]
The quantity \(\sigma\) vanishes on the rotating soliton (see the appendix).

The square-profile initial curves \(\cC_j(-j)\) consist of two arcs.  One is the Yin-Yang soliton, so on this arc we have \(\sigma=0\).  The other arc is the radial line segment on the ray \(\theta=\vartheta_*^+(-j)\).  On this segment we clearly have \(\kappa=0\).  Since we orient \(\cC_j\) counterclockwise, \(X\) and \(X_s\) are parallel with opposite directions; i.e.~\(-\langle X, X_s\rangle > 0\).  Hence \(\sigma>0\) on the line segment.  Finally, the initial curve \(\cC_j(-j)\) is not smooth, having two corners where the line segment and Yin-Yang arms meet.  If one rounds these corners off by replacing them with small circle arcs with radius \(\rho\ll1\), then the curvature of these arcs will be \(\kappa=\rho^{-1}\gg1\), so that \(\sigma>0\) on the circular arcs, provided \(\rho\) is sufficiently small.  The resulting curve has \(\sigma=0\) on the Yin-Yang arms, and \(\sigma>0\) on the line segment, as well as the small circular arcs.  The solution to CS starting from the modified initial curve therefore has \(\sigma>0\).  Letting \(\rho\searrow 0\) we conclude that \(\sigma>0\) also holds on \(\cC_j(t)\).
\end{proof}

With Theorem \ref{sturmian}, we can decompose the solutions \(\cC_j(t)\) into exactly two graphs over the polar angle parameter.

\begin{lem}\label{lem-polargraphdecomp}
For any \(t \in (-j, -T^*]\), there is an interval \([\vt^-_j(t), \vt^+_j(t)]\) such that the curve \(\cC_j(t)\) can be written as the union of two graphs of polar functions, \(R_j^-(\theta, t)\) and \(R_j^+(\theta, t)\) defined for \(\theta \in [\vt^-_j(t), \vt^+_j(t)]\).
The functions \(t\mapsto \vartheta^-_j(t)\) and \(t\mapsto \vartheta^+_j(t)\) are strictly increasing and decreasing, respectively.
\end{lem}

\begin{proof} By the maximum principle, the ``really old solutions'' \(\cC_j(t)\) will be contained inside of
the Yin-Yang curve.  The Sturmian property, Theorem \ref{sturmian}, tells us that the
number of intersections of \(\cC_j(t)\) and the rays \(\theta = \theta_0 \in \R\) is
non-increasing, and only decreases when there is a tangency.  This implies that the desired graph decomposition exists.  These two graphs are bounded above and below by the branches of the Yin-Yang soliton on their polar interval of definition, \([\vt^-_j(t), \vt^+_j(t)]\).
\end{proof}

Similarly, we can always write each \(\cC_j(t)\) as a union of two graphs taking
values in \(\theta\), the polar angle.  Recall that the images of \(\cR(\theta
-t)\) for \(t \in [-\pi,\pi)\) foliate the punctured plane \(\R^2 \setminus \{0\}\). See Figure~\ref{fig-foliation}.

\begin{lem}
For all \(t\), \(\cC_j(t)\) can be decomposed into two graphs of two functions which
take leaves of the foliation as inputs and have their range in the set of polar
angles.  More specifically, for \(T\ll 0\) there exist \(y_{j, 1}, y_{j,
2}:(-\infty, T] \to \big(-\tfrac{\pi}{2}, \tfrac{\pi}{2}\big)\) and functions 
\[
\Theta_j^\pm:\{(t, y) : t<T, y_{j, 1}(t) < y < y_{j, 2}(t)\} \to \R
\]
such that the very old solution \(\cC_j(t)\) is the union of the two curves
\[
Y(\Theta_j^\pm(t, y), t, y) = \cR\bigl( \Theta_j^\pm(t, y) - t + y\bigr) \vE1(\theta),
\]
where $Y$ is given by \eqref{eq-YY-leaves-parametrized}.
\end{lem}

\begin{proof}
The initial square-profile curve \(\cC_j(-j)\) is tangent to the graphs of
\(\cR(\theta \pm \pi/2 +j)\) and intersects the graphs of \(\cR(\theta + y + j)\),
\(y \in (-\pi/2,\pi/2)\) twice: once at the origin and once on the line segment
connecting the two branches of \(\cC_j(-j).\) Then, by the Sturmian theorem, for
all subsequent \(t>-j\), \(\cC_j(t)\) can be split into two graphs corresponding to
the ``upper'' and ``lower'' intersection points with the leaves of the Yin-Yang
foliation. At each time \(t\), these graphs split at two unique leaves of the
foliation, marked by values \(y_{j,1}(t), y_{j,2}(t) \in (-\pi/2, \pi/2)\), so
that \(\cC_j(t)\) is tangent to the curves \(\{r = \cR(\theta + y_{j,1}(t) - t)\}\)
and \(\{r = \cR(\theta + y_{j,2}(t) - t)\}\).  We know that these two points are
unique since a greater number of tangencies would introduce more than two
intersection points for other curves \(\{r = \cR(\theta + y - t)\}\).  We call the
coordinate system \((y, \theta) \in (-\pi, \pi) \times (0, \infty)\) the
``Yin-Yang polar coordinate system'' and denote the two functions giving the
upper and lower graphs comprising \(\cC_j(t)\) by \(\Theta_j^-(y,t)\) and
\(\Theta_j^+(y,t)\) respectively, defined on the interval \((y_{j,1}(t),y_{j,2}(t))
\subset (-\pi/2, \pi/2)\).
\end{proof}

\begin{lem}\label{linftytip}
  There exist \(T<0\) and \(C>0\) such that \(\vartheta_j^+(t) \leqslant \vartheta_*^+(t)+C\) for all \(j\in\N\) and all \(t \in [-j, -T]\).
\end{lem}

\begin{proof}
Assume that \(\epsilon < \pi / 16\).  

For any \(t\in[-j, T]\) at which \(\vartheta_j^+(t) > \vt_*^+(t)\) we consider the area \(\cA_j(t)\) of the ``really old solution'' \(\cC_j(t)\) inside the polar interval \([\vt_*^+(t),\vartheta_j^+(t)]\), where \(\vt_*^+(t)\) and \(\vartheta_j^+(t)\) are the endpoints of the intervals of definition of the approximate solution \(\cC_*(t)\) and \(\cC_j(t)\) respectively.  This area measures the ``tail'' of the \(\cC_j(t)\) that may form between the tip of \(\cC_j(t)\) and the tip of \(\cC_*(t)\).  Note that the area \(\cA_j(t)\) is bounded above by the error
\[ \cA_j(t) \leqslant A_j + \cE(T^*) = A_j + \int_{-\infty}^{-T^*} \int_{\cC_*(\tau)} |V - \kappa| ds \; d\tau < \epsilon
\]
To calculate this area, first consider the function \(\Theta_{t,j}(y) := \max\{\Theta_j^+(y,t) - \vartheta_*^+(t), 0\}\) over the interval \((y_{j,1}(t),y_{j,2}(t))\).  Then in the \((\theta, y)\) ``Yin-Yang coordinates,'' we can integrate to find the area:
\[\cA_j(t) = \int_{y_{j,1}(t)}^{y_{j,2}(t)} \int_{\vartheta_*^+}^{\Theta_{t,j}(y) + \vartheta_*^+} \cR(\theta +y - t)\det(D\mathcal{T}) d\theta dy,
\]
where \(\mathcal{T}: (0, \infty) \times (-\pi/2, \pi/2) \rightarrow (0, \infty) \times (0, \infty)\) is the coordinate transformation given by \(\mathcal{T}(\theta, y) = (\theta, \cR(\theta + y - t))\).  Clearly, \(\det \mathcal{T}= \cR' (\theta + y - t)\), so
\[\cA_j(t) = \int_{y_{j,1}(t)}^{y_{j,2}(t)} \int_{\vartheta_*^+}^{\Theta_{t,j}(y) + \vartheta_*^+} \cR(\theta +y - t)\cR'(\theta +y - t) d\theta dy \approx \int_{y_{j,1}(t)}^{y_{j,2}(t)} \Theta_{t,j}(y) dy,
\]
by the asymptotic expansions in \eqref{eq-YY-simple-asymptotics}.

We argue that given a small \(\delta >0\), it is possible to pick an angle \(\theta_0\) independent of \(j\) such that the measure \(|\{y:\Theta_{t,j}(y)>\theta_0\}| < \delta\).
Indeed, it follows from
\[
\theta_0 \big|\{y:\Theta_{t,j}(y)>\theta_0\}\big| \leqslant \int_{y_{j,1}(t)}^{y_{j,2}(t)} \Theta_{t,j}(y) dy < \epsilon
\]
that if \(\theta_0 < \frac{\epsilon}{\delta}\), then  \(|\{y:\Theta_{t,j}(y)>\theta_0\}| < \delta\) holds for all \(t, j\).

\medskip

The two points intersection of \(\cC_j(t)\) with the ray \(\theta=\vartheta_*^+(t)+\theta_0\) are
\[
P^\pm(t) = R_j^\pm(\vartheta_*(t)+\theta_0, t).
\]
Let \(\gamma(t)\) be the arc on \(\cC_j(t)\) on which \(\theta\geq \vartheta_*^+(t)+\theta_0\), and whose endpoints therefore are \(P^\pm(t)\).  Consider the area \(A(t)\) of the region enclosed by \(\gamma(t)\) and the line segment connecting \(P^\pm(t)\).  This area changes because the arc \(\gamma(t)\) moves, and also because the line segment \(P^-P^+\) moves.  The rate of change is therefore the sum of \(-\int_{\gamma(t)}\kappa ds\) and the rate at which the segment \(P^-P^+\) sweeps out area.

\begin{figure}[h]
  \centering
  \includegraphics[width=\textwidth]{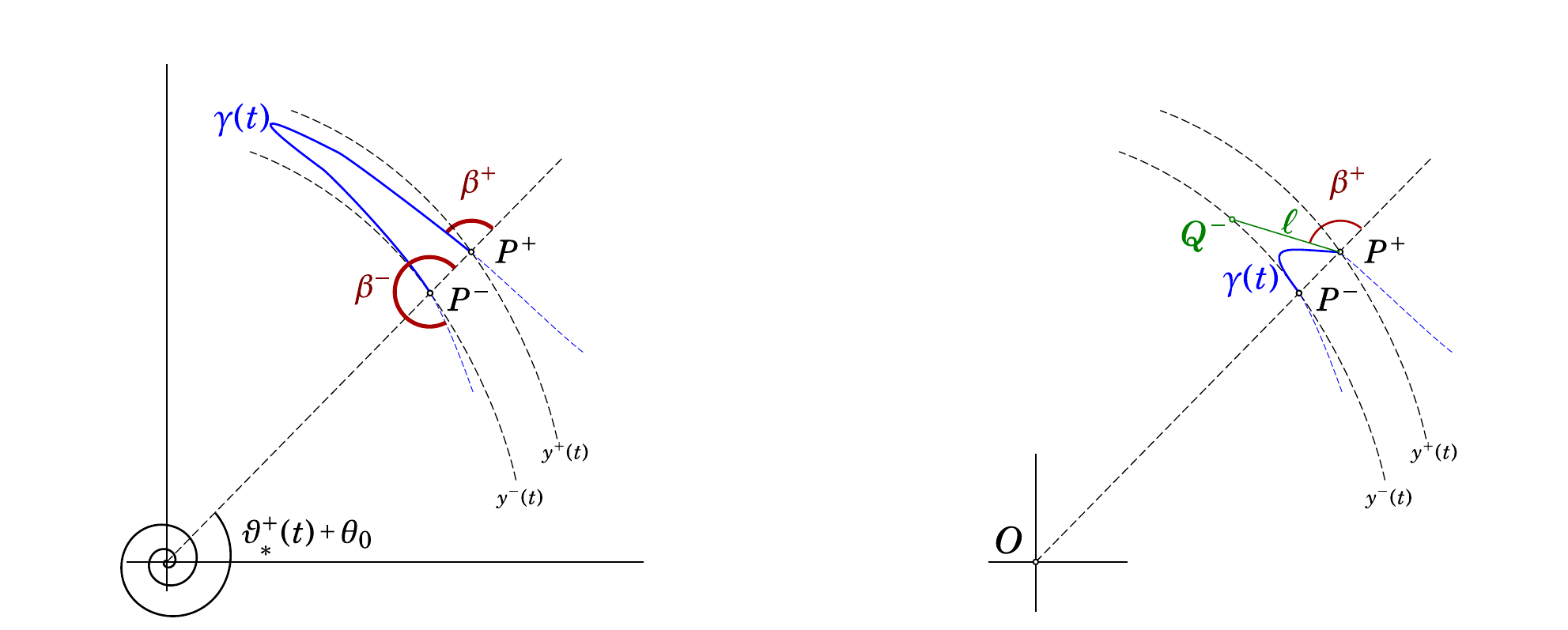}
  \caption{\textbf{Left: }The arc \(\gamma(t)\).  \textbf{Right: } the angles \(\beta^\pm\)}
  \label{fig-phiEstimate}
\end{figure}

If \(\phi:\gamma(t)\to\R\) is the tangent angle along the arc (i.e.~\(X_s = \vE1(\phi)\)), then the curvature integral is
\[
\int_{\gamma(t)}\kappa\,ds = \phi_{P^-(t)} - \phi_{P^+(t)}.
\]
The line segment \(P^+P^-\) moves with angular velocity \(\frac{d}{dt}\vartheta_*^+(t)\) and therefore adds area to the region enclosed by \(\gamma(t)\) at the rate
\[
\frac 12\Bigl\{\bigl(R_j^+\bigr)^2 - \bigl(R_j^-\bigr)^2 \Bigr\}
\frac{d\vartheta_*^+(t)}{dt} 
\]
in which \(R_j^\pm\) are evaluated at \(\theta=\vartheta_*^+(t)+\theta_0\).
Our construction of the cap implies that \(\vartheta_*^+(t) = -t + o(1)\), and that this relation may be differentiated: \(\frac{d}{dt}\vartheta_*^+(t) =  -1 + o(t)\).

The radii \(R_j^\pm(\vartheta_*^+(t)+\theta_0, t)\) are given in terms of their Yin-Yang coordinates \(y^\pm(t)\) via
\[
R_j^\pm(\vartheta_*^+(t)+\theta_0, t) = \cR(\vartheta_*^+(t)+\theta_0 - t + y^\pm).
\]
It follows that at \(\vartheta_*^+(t)+\theta_0\) 
\begin{align*}
\frac 12\Bigl\{\bigl(R_j^+\bigr)^2 - \bigl(R_j^-\bigr)^2 \Bigr\}
  &= \frac12 \bigl\{\cR(\vartheta_*^+(t)+\theta_0 - t + y^+)^2 -  \cR(\vartheta_*^+(t)+\theta_0 - t + y^-)^2\bigr\} \\
  &= \cR\,\cR'\,(y^+-y^-)
\end{align*}
in which \(\cR, \cR'\) are evaluated at \(\vartheta_*(t)+\theta_0+\tilde y\) for some \(\tilde y\in[y^-, y^+]\) that is provided by the mean value theorem.  The asymptotics of \(\cR\) imply that \(\cR\cR' = 1+o(1) <2\).  Our choice of \(\theta_0\) was such that  \(0<y^+-y^-\leqslant\delta\).  Hence 
\[
\left|
\frac 12\Bigl\{\bigl(R_j^+\bigr)^2 - \bigl(R_j^-\bigr)^2 \Bigr\}
\frac{d\vartheta_*^+(t)}{dt}
\right| \leqslant 2\delta.
\]
In total, the rate at which the area \(A(t)\) enclosed by the arc \(\gamma(t)\) grows is bounded by
\[
\frac{dA}{dt} \leqslant  -\bigl(\phi_{P^-(t)} - \phi_{P^+(t)}\bigr) + 2\delta.
\]
\medskip

We estimate the change in tangent angle across the arc \(\gamma(t)\). 
Let \(\beta^+\) be the counterclockwise angle from the ray \(\theta=\vartheta_*^+(t)+\theta_0\) to the tangent \(X_s\) to \(\gamma\) at \(P^+\), and similarly, let \(\beta^-\) be the counterclockwise angle from the same ray to the tangent to \(\gamma\) at \( P^-\) (see Figure~\ref{fig-phiEstimate}).  We have \(0<\beta^+<\pi<\beta^-\) and \(\phi_{P^-} - \phi_{P^+} = \beta^-  - \beta^+\).

Recall that \(\kappa - \langle X, X_s \rangle > 0\) along \(\cC_j(t)\).  Since \(\kappa=\phi_s\) and \(\langle X, X_s \rangle = \frac12 \frac{d}{ds}r^2\), where \(r=\|X\|\), it follows that \(\phi-\frac12 r^2\) increases as one traverses \(\gamma\) from \(P^+\) to \(P^-\).  Thus, at any point with polar coordinates \((\theta, r)\) on \(\gamma\) one has 
\[
\phi > \phi_{P^+} + \frac12 (r^2 - r_{P^+}^2)
\]
The lowest value \(r\) has on \(\gamma\) occurs at the point \(P_-\), and we have just shown that \(r_{P^+}^2 - r_{P_-}^2 \leqslant 2\delta\).  Hence \(\phi>\phi_{P^+}-2\delta\) on the entire arc \(\gamma\).

It follows that if \(\beta^+>\frac 34\pi\), then the angle between the tangent to \(\gamma\) and the ray \(OP^+\) (\(O\) is the origin) will always be at least \(\frac34\pi-2\delta\), i.e.~more than \(\frac58\pi\), provided we choose \(\delta<\frac{\pi}{16}\).  

Consider the line \(\ell\) through \(P^+\) whose angle with \(OP^+\) is \(\frac58 \pi\).  The euclidean distance between \(P^-\) and \(P^+\) is \(r_{P^+}-r_{P^-} \leqslant 2\delta/(r_{P^-}+r_{P^+})\leqslant C\delta |t|^{-1/2}\), since \(r_{P^+} > r_{P^-}\gtrsim |t|^{1/2}\).

At this scale the Yin-Yang leaves will be almost straight lines near \(P^\pm\), so that the line \(\ell\) then intersects the Yin-Yang leaf with \(y=y_-(t)\) at a point \(Q^-(t)\), also at a distance \(d(P^-, Q^-) \lesssim \delta|t|^{-1/2}\).

\[
  \vartheta_j^+(t) - \vartheta_*^+(t)-\theta_0 
  \sim \frac{d(P^-, Q^-)}{r_{P^-}} \lesssim \delta |t|^{-1} \ll \theta_0.
\]
Hence the largest polar angle on \(\gamma\) will be at most
\[
\vartheta_j^+(t) \leqslant \vartheta_*^+(t) + \theta_0 + \delta |t|^{-1}
\leqslant \vartheta_*^+(t) + 2\theta_0 .
\]
Thus we find that if \(|t|\) is sufficiently large, then either \(\vartheta_j(t) < \vartheta_*(t)\), or else \(\beta^+<\frac58\pi\).  In the latter case the area enclosed by \(\gamma(t)\) decreases faster than
\[
\frac{dA}{dt} \leqslant - \beta^-+\beta^+ + 2\delta \leqslant
-\pi + \frac58\pi + 2\delta = - \frac38\pi+2\delta < -\frac\pi4,
\]
again assuming that \(\delta<\pi/16\).

We now finally prove that \(\vartheta_j^+(t)-\vartheta_*^+(t)\) is uniformly bounded for all \(t\in[-j, T]\) and~\(j\).

At \(t=-j\) we have \(\vartheta_j^+(t) < \vartheta_*^+(t)+\theta_0\), by definition of the initial curve \(\cC_j(-j)\).  Hence, if at any time \(t_1<T\) one has \(\vartheta_j^+(t) > \vartheta_*^+(t)+2\theta_0\), then there is a largest interval \((t_2, t_3)\ni t_1\) on which  \(\vartheta_j^+(t) > \vartheta_*^+(t)+2\theta_0\).  In particular, at \(t=t_2\) one has  \(\vartheta_j^+(t) = \vartheta_*^+(t)+2\theta_0\).

Define the arc \(\gamma(t)\) as above.  Its enclosed area is at most \(\varepsilon\), where we may assume that \(\varepsilon<\pi/4\).  During the time interval \((t_2, t_3)\) the area decreases at a rate of at least \(\pi/4\), and therefore the length \(t_3-t_2\) of the time interval cannot exceed \(\varepsilon/(\pi/4) < 1\).  At time \(t=t_2\) we had \(\vartheta_j^+(t) = \vartheta_*^+(t)+2\theta_0\).  Since \(\vartheta_j^+(t), \vartheta_*^+(t)\) are nonincreasing functions, we have throughout \((t_2, t_3)\)
\begin{align*}
\vartheta_j^+(t) - (\vartheta_*^+(t)+2\theta_0)
&\leqslant \vartheta_j^+(t_2) - (\vartheta_*^+(t)+2\theta_0)\\
&= \vartheta_*^+(t_2) + 2\theta_0 - (\vartheta_*^+(t)+2\theta_0)\\
&\leqslant \vartheta_*^+(t_2) + 2\theta_0 - (\vartheta_*^+(t_3)+2\theta_0)\\
&\leqslant  \vartheta_*^+(t_2) -\vartheta_*^+(t_3).
\end{align*}
Since \(\frac{d}{dt}\vartheta_*^+(t) = 1+ o(1)\) we find that
\[
\vartheta_j^+(t) - (\vartheta_*^+(t)+2\theta_0) \leqslant 1+o(1) <2
\]
for all \(t\in(t_2, t_3)\).
\end{proof}


To summarize, we can now decompose every very-old solution \(\cC_j(t)\), for
\(t \in [-T^*- 1/2, -T^*]\) into four graphs in two different coordinate systems,
\(R_j^\pm (\theta, t)\) in polar coordinates, and \(\Theta^\pm_j(y,t)\) in Yin-Yang polar
coordinates.

\begin{lem}[Curvature bounds]
There exist \(T<0\) such that for any \(T'<T\) the lengths \(L_j(t)\) and curvatures of \(\cC_j(t)\) are uniformly bounded for all \(j\) and all \(t\in[T', T-2]\).
\end{lem}
\begin{proof}
The length bounds follow from the fact that in \((y, \theta)\) coordinates each \(\cC_j(t)\) is contained in a uniformly bounded rectangle \(|y|\leqslant \pi/2\), \(-t\leqslant \theta\leqslant \vartheta_*^+(t)+2\theta_0+2\), and the fact that \(\cC_j(t)\) decomposes into four segments on each of which both \(y\) and \(\theta\) are monotone.

Consider a given value \(T'<T\).  Assume that our Lemma fails, and that along some subsequence \(j\) the maximal curvature of \(\cC_j(t)\) with \(t\in[T', T-2]\) becomes unbounded.

For \(t\in[T'-2, T]\) the lengths \(L_j(t)\) of \(\cC_j(t)\) are uniformly bounded by some \(L>0\).  It follows that
\[
\int_{T'-2}^{T'-1}\int_{\cC_j(t)} \kappa^2\, ds\, dt = \bigl[-L_j(t)\bigr]_{T'-2}^{T'-1} < L.
\]
Therefore, there is a sequence \(t_j\in [T'-2, T'-1]\) such that
\[
\int_{\cC_j(t_j)} \kappa^2\, ds < L.
\]
By a Sobolev embedding theorem this implies that the curves \(\cC_j(t_j)\) are uniformly \(C^{1, 1/2}\), i.e.~they are continuously differentiable, and their tangent angles \(\phi_j\) are uniformly Hölder continuous~---~in fact, for any two points at arclength coordinates \(s_1, s_2\) in \(\cC_j(t_j)\) one has
\[
|\phi(s_2)-\phi(s_1)| = \left|\int_{s_1}^{s_2} \kappa\, ds\right| \leqslant \sqrt{s_2-s_1} \sqrt{\int_{s_1}^{s_2}\kappa^2ds} \leqslant \sqrt L \sqrt{s_2-s_1}.
\]
It follows that all \(\cC_j(t_j)\) are uniformly locally Lipschitz curves.  Now consider the solutions to curve shortening with \(\cC_j(t_j)\) as initial data, i.e.~consider \(\tilde\cC_j(t) = \cC_j(t_j+t)\).  These solutions all exist for \(0\leqslant t\leqslant T-t_j\geqslant T-T'\).  Supposing that along some subsequence of \(t_j\) the curvatures of the \(\tilde\cC_j\) are not bounded for \(1\leqslant t\leqslant T-T'\), we pass to a further subsequence for which the initial curves \(\tilde\cC_j(0)\) converge in \(C^1\) to some limit curve \(\tilde \cC_*\).  The enclosed areas of the \(\tilde\cC_j(0)\) then also converge, and hence, by Grayson's theorem~\cite{grayson:hes86}  the evolution by Curve Shortening \(\cC_*(t)\) starting from \(\cC_*\) exists for \(0\leqslant t\leqslant T-T'\).  By continuous dependence on initial data it follows that the solutions \(\tilde\cC_j(t)\) converge in \(C^\infty\) to \(\cC_*(t)\) on any time interval \([\delta, T-T']\) with \(\delta>0\).  This implies that the curvatures of the \(\tilde \cC_j(t)\) are uniformly bounded for \(t\in [1, T-T']\), which then implies that the curvatures of \(\cC_j(t)\) are uniformly bounded after all for \(t\in[t_j+1, t_j+T-T'] \subset [T', T-2]\).
\end{proof}

\appendix
\section{The Yin-Yang Soliton}\label{sec-YY-properties}
Hungerbühler and Smoczyk \cite{2000HungerBuehlerSmoczyk} proved uniqueness and existence of a rotating soliton for curve shortening that contains the origin.  See also Halldorson~\cite{2010Halldorsson} and Altschuler~et.al.~\cite{MR3043378}.  Here we derive its more detailed asymptotic behavior, which we use in our construction of the approximate solution.

\subsection{The Yin--Yang soliton in polar coordinates}

For an evolving family of curves written in polar coordinates \( X(t, \theta) = r(t, \theta) \vE1(\theta)\) the Curve Shortening Deficit is
\[
(V-\kappa)\,ds = \left\langle X_t - \tfrac{X_{\theta\theta}}{\|X_\theta\|^2}, JX_\theta\right\rangle d\theta =\Biggl\{-rr_t +\frac{r(r_{\theta\theta}-r)-2r_\theta^2} {r^2+r_\theta^2}\Biggr\}d\theta.
\]
It follows that \(X\) is a solution of CSF if and only if \(r(\theta,t)\) satisfies
\begin{equation}
\label{eq-CSF-polar-coords}
r\frac{\partial r}{\partial t} =
\frac{rr_{\theta\theta} - r_\theta^2}{r^2+r_\theta^2} - 1
= \frac{\partial}{\partial\theta}\Bigl(\arctan \frac{r_\theta}{r}\Bigr) - 1.
\end{equation}

If we look for solutions of the form \(r(\theta,t) = \cR(\theta - t)\) we get an ODE for \(\cR(\theta)\)
\begin{equation}
\label{eq-YY-ode}
-\cR\cR' = \frac{\cR\cR''-\cR'^2}{\cR^2+\cR'^2}-1.
\end{equation}
Hungerbühler and Smoczyk~\cite{2000HungerBuehlerSmoczyk} observed that this equation can be integrated once.  By suitably rotating the curve around the origin we can ensure that the resulting integration constant vanishes, and we therefore have
\begin{equation}\label{HunSmo}
\frac{1}{2}\cR(\theta)^2 - \theta
+ \arctan\frac{\cR'(\theta)}{\cR(\theta)}
=0.
\end{equation}
We consider the soliton that passes through the origin.  When this happens \(\cR\to0\) and \(\cR'\to\infty\), so that~\eqref{HunSmo} implies \(\theta\to\pi/2\).  The function \(\cR(\theta)\) is therefore defined for all \(\theta>\pi/2\), and, as proved by Hungerbühler and Smoczyk, \(\cR'(\theta)>0\) for all \(\theta>\pi/2\).

\subsection{Asymptotic expansion of \(\cR\)}

We now show that \(\cR(\theta)\) has an asymptotic expansion of the form
\begin{equation}
\label{eq-R-expansion-general}
\cR(\theta) = (2\theta)^{1/2}\left\{ 1 + \frac{c_1}{2\theta} + \frac{c_2}{(2\theta)^{2}}
+\cdots+ \frac{c_{N}}{(2\theta)^{N}}+ \cO(\theta^{-N-1})\right\}
\qquad (\theta\to\infty)
\end{equation}
for any \(N\in\N\).  These expansions can be differentiated any number of times.  The coefficients \(c_j\) can be computed by substituting the expansions in \eqref{eq-YY-ode} and recursively solving for \(c_j\).  In particular, one finds \(c_1=0\), \(c_2=-1\), and \(c_3=\frac{11}{3}\) so that
\begin{align}
  \label{eq-R-expansion-two-terms}
  \cR(\theta)
  &= \sqrt{2\theta} - \frac{1}{(2\theta)^{3/2}}
    + \frac{11}{3} \frac{1}{(2\theta)^{5/2}} + \cO(\theta^{-7/2})
  &&(\theta\to\infty) \\
  \label{eq-Rderiv-expansion-two-terms}
  \cR'(\theta)
  &= \frac{1}{\sqrt{2\theta}} + \frac{3}{(2\theta)^{5/2}}
    - \frac{55}{3} \frac{1}{(2\theta)^{7/2}}  + \cO(\theta^{-9/2})
  &&(\theta\to\infty).
\end{align}
This implies that the quantities \(R=R(t)=\cR(-2t)\), \(\varepsilon=1/R\), and \(R_\theta=R_\theta(t)=\cR'(-2t)\), which we use in the construction of the cap, have expansions in powers of \(\tau=-4t\), given by
\begin{equation}
\label{eq-R-etc-expanded}
\left\{\quad
\begin{aligned}
R = \cR(-2t)
&= \sqrt\tau \bigl\{1-\tau^{-2} + \tfrac{11}{3}\tau^{-3}+ \cO(\tau^{-4})\bigr\}\\
\varepsilon = \frac1R &=
\frac{1}{\sqrt\tau} \bigl\{1 + \tau^{-2} -\tfrac{11}{3}\tau^{-3} + \cO(\tau^{-4})\bigr\}.  \\
R_\theta = \cR'(-2t) &= \frac{1}{\sqrt\tau} \bigl\{1+3\tau^{-2} - \tfrac{55}{3}\tau^{-3} + \cO(\tau^{-4})\bigr\}.
\end{aligned}
\right.
\end{equation}

\medskip

\textit{Proof of \eqref{eq-R-expansion-general}} Consider the quantity \( u(\theta) = \theta - \tfrac{1}{2} \cR(\theta)^2\).  Since \(\cR'(\theta)>0\) for all \(\theta\) it follows from \eqref{HunSmo} that
\begin{equation}
\label{eq-u-basic-bound}
0< u(\theta)<\frac\pi2\text{ for all }\theta>\frac\pi2.
\end{equation}
Directly differentiating \(u=\theta-\frac12\cR^2\) and using \eqref{HunSmo} we find an equation for \(u\),
\begin{equation}
\label{eq-u-ode}
u'(\theta) = 1 - 2(\theta - u(\theta)) \tan u(\theta).
\end{equation}
i.e.
\[
u'+ 2\theta u = F(\theta,u) \isdef 1+2\theta (u-\tan u) + 2u\tan u
\]
We use induction to show that \(u(\theta)\) has an expansion of the form
\begin{equation}\label{eq-u-expansion}
u(\theta) =
\frac{u_1}\theta + \frac{u_2}{\theta^2}+\cdots+ \frac{u_N}{\theta^N}+\cO(\theta^{-N-1})
\end{equation}
for any \(N\in \N\).

Begin with the case \(N=0\).  We know that \(0<u<\pi/2<\theta\), so that \(u<\tan u\), and hence
\[
u'(\theta) = 1 - 2(\theta - u) \tan u \leqslant 1 -2 (\theta - u) u = 1 - 2 \theta u + 2u^2
\]
which implies
\[
u' + 2 \theta u \leqslant 1+2u^2 \leqslant 1+\frac{\pi^2}{2}=C_0.
\]
Multiply with \(e^{\theta^2}\), and integrate
\begin{align*}
  0<u(\theta)
  \leqslant
  u(\theta_0)e^{\theta_0^2-\theta^2}
  + C_0 \int_{\theta_0}^{\theta} e^{\xi^2-\theta^2} \, d\xi
  =O(\theta^{-1}) \quad(\theta\to\infty),
\end{align*}
Thus the case \(N=0\) holds.

For the induction step we expand \(\tan u\) in a Taylor series,
\[
\tan u = u + \frac{u^3}3 + \cdots = u + \sum_{k\geqslant 1}c_k u^{2k+1}
\]
and rewrite the equation for \(u\) as
\[
F(\theta, u) = 1 + 2u^2 + \sum_{k\geqslant1}2(\theta-u)c_ku^{2k+1}
\]
Multiplying with \(e^{\theta^2}\) and integrating from some fixed \(\theta_0>\pi/2\) we get
\begin{equation}\label{eq-u-integral-eqn}
u(\theta) = e^{\theta_0^2-\theta^2}u(\theta_0)
+ \int_{\theta_0}^\theta e^{\xi^2-\theta^2}F(\xi, u(\xi))\, d\xi
\end{equation}
Repeated integration by parts leads to
\begin{multline}\label{eq-int-xik-by-parts}
\int_{\theta_0}^\theta e^{\xi^2-\theta^2} \xi^{-k}d\xi =\\
\frac12 \theta^{-k-1}+\frac{k+1}{2}\theta^{-k-3}+\cdots +\frac{(k+1)\cdots(k+2m-1)}{2^m}\theta^{-k-2m-1} + \cO(\theta^{-k-2m-3})
\end{multline}
for all \(k,m\in\N\).  If we assume that \(u\) has an expansion up to \(\cO(\theta^{-(N-1)})\), then we also have expansions for \(u^2\) and \(F(\theta, u)\) up to \(\cO(\theta^{-(N-1)})\).  Substitute these expansions in the integral equation~\eqref{eq-u-integral-eqn} and use~\eqref{eq-int-xik-by-parts} to conclude that \(u\) has an expansion up to \(\cO(\theta^{-N})\), as claimed.

The expansion~\eqref{eq-u-expansion} for \(u\), which we now have proved, implies \(\cR(\theta)=\sqrt{2(\theta-u(\theta))}\) also has an asymptotic expansion in powers of \(\theta^{-1}\).

Finally, while one cannot in general differentiate asymptotic expansions, one can integrate them.  Thus if a function \(f(\theta)\) and its derivative \(f'(\theta)\) both have asymptotic expansions in powers of \(\theta^{-1}\), then by integrating the expansion of \(f'\) one should get the expansion for \(f\), up to a constant: this implies that the expansion of \(f'\) can be found by differentiating the expansion for \(f\).  We therefore only have to show that all derivatives of \(u(\theta)\) have expansions in powers of \(\theta^{-1}\), which will then imply that the expansions~\eqref{eq-u-expansion} can be differentiated.

To find expansions for \(u^{(m)}\), note that if \(u\) has an expansion with remainder \(\cO(\theta^{-N-1})\), then simple substitution in the differential equation~\eqref{eq-u-ode} leads to an expansion for \(u'(\theta)\) with remainder \(\cO(\theta^{-N})\).  Going further, one can differentiate~\eqref{eq-u-ode} \(m-1\) times and express \(u^{(m)}(\theta)\) in terms of \(u, u', u'', \dots, u^{(m-1)}\).  This implies that if one has an expansion in powers of \(\theta^{-1}\) of the first \(m-1\) derivatives of \(u\), then one also has an expansion for \(u^{(m)}\).  By induction it follows that all derivatives of \(u\) have such expansions.

Similar arguments also apply to the expansions of \(\cR(\theta)\).

\subsection{Inversion of the expansion of \(\cR\)}\label{sec:cR-inverted}
The expansion \eqref{eq-R-expansion-general}, which expresses \(\cR\) as a function of \(\theta\), implies that one can invert the function \(\theta\mapsto\cR(\theta)\), and that the inverse has an asymptotic expansion.  It follows from \eqref{eq-R-expansion-general} that
\[
\cR(\theta)^2 = (2\theta)\left\{ 1 + \frac{\bar c_1}{2\theta} + \frac{\bar c_2}{(2\theta)^{2}} +\cdots+ \frac{\bar c_{N}}{(2\theta)^{N}}+ \cO(\theta^{-N-1})\right\}
\]
and hence
\[
2\theta = \cR^2 \left\{ 1 + \frac{\bar c_1}{2\theta} + \frac{\bar c_2}{(2\theta)^{2}} +\cdots+ \frac{\bar c_{N}}{(2\theta)^{N}}+ \cO(\theta^{-N-1})\right\}^{-1}.
\]
Repeated substitution of this expansion in itself allows one to convert all powers of \((2\theta)\) on the left into powers of \(\cR^2\), so that we have an expansion
\[
2\theta = \cR^2 \left\{ 1 + \frac{\tilde c_1}{\cR^2} + \frac{\tilde c_2}{\cR^4} +\cdots+ \frac{\tilde c_{N}}{\cR^{2N}}+ \cO(\cR^{-2N-2})\right\}
\]
for certain coefficients \(\tilde c_i\).

\bibliographystyle{plain} 
\bibliography{GApapers.bib}

\end{document}